\documentclass[11pt]{amsart}%

\usepackage{amsmath, amssymb}

\usepackage{amsthm}

\usepackage{color}

\usepackage{graphicx}

\usepackage{indentfirst}

\makeatletter
\def\widebreve#1{\mathop{\vbox{\m@th\ialign{##\crcr\noalign{\kern\p@}%
  \brevefill\crcr\noalign{\kern0.1\p@\nointerlineskip}%
  $\hfil\displaystyle{#1}\hfil$\crcr}}}\limits}
\makeatother

\makeatletter
\def\brevefill{$\m@th \setbox\z@\hbox{}%
 \hfill\scalebox{0.7}{\rotatebox[origin=c]{90}{(}} \kern4pt $}
\makeatother

\theoremstyle{plain}
\newtheorem{theo}{Theorem}[subsection]

 \makeatletter
    
    \@addtoreset{equation}{section}
  \makeatother

\newtheorem{lemm}[theo]{Lemma}
\newtheorem{pro}[theo]{Proposition}
\newtheorem{cor}[theo]{Corollary}

\theoremstyle{remark}
\newtheorem{rem}[theo]{Remark}
\newtheorem{ex}[theo]{Example}

\theoremstyle{definition}
\newtheorem{defi}[theo]{Definition}

\newtheorem*{claim}{Claim}
\newtheorem{nclaim}{Claim}

\newcommand{\B}{{\mathcal{B}}}

\newcommand{\ed}{\mathrm{end}}
\newcommand{\id}{e}

\newcommand{\wt}{\mathrm{wt}}
\newcommand{\dg}{\mathrm{deg}}

\newcommand{\dir}{\mathrm{dir}}
\newcommand{\OS}{\mathrm{OS}}
\newcommand{\aff}{\mathrm{aff}}
\newcommand{\ext}{\mathrm{ext}}
\newcommand{\eqdef}{:=}
\newcommand{\reqdef}{=:}
\newcommand{\ardef}{\overset{\rm def}{\Leftrightarrow}}
\newcommand{\bqed}{\quad \hbox{\rule[-0.5pt]{6pt}{6pt}}  \vspace{3mm}}
\newcommand{\lon}{w_\circ}
\newcommand{\lons}{w_\circ (S)}

\newcommand{\zero}{ {\bf 0 } }
\newcommand{\DBG}{\mathrm{DBG}}
\newcommand{\PQLS}{\mathrm{pQLS}}
\newcommand{\BPQLS}{\overline{\mathrm{pQLS}}}
\newcommand{\half}{\frac{1}{2}}
\newcommand{\pr}{\mathrm{pr}}

\newcommand{\pair}[2]{\langle #1,\,#2 \rangle}

\newcommand{\BZ}{\mathbb{Z}}

\newcommand{\BB}{\mathbb{B}}

\newcommand{\ve}{\varepsilon}
\newcommand{\vp}{\varphi}

\usepackage{color}

\newenvironment{enu}{%
 \begin{enumerate}%
}{\end{enumerate}}

\usepackage[left=1in, right=1in, top=1in, bottom=1in]{geometry}

\begin{document}

\setlength{\parskip}{4pt}

\title[Macdonald Polynomials via a Path Model]{Symmetric and Nonsymmetric Macdonald Polynomials via a Path Model with a Pseudo-crystal Structure}

\author[C.~Lenart]{Cristian Lenart}
\address[Cristian Lenart]{Department of Mathematics and Statistics, State University of New York at Albany, 
Albany, NY 12222, U.S.A.}
\email{clenart@albany.edu}

\author[S.~Naito]{Satoshi Naito}
\address[Satoshi Naito]{Department of Mathematics, Tokyo Institute of Technology,
2-12-1 Oh-Okayama, Meguro-ku, Tokyo 152-8551, Japan}
\email{naito@math.titech.ac.jp}

\author[F.~Nomoto]{Fumihiko Nomoto}
\address[Fumihiko Nomoto]{Tokyo Tech High School of Science and Technology,
3-3-6 Shibaura, Minato-ku, Tokyo 108-0023, Japan}
\email{fnomoto@hst.titech.ac.jp}

\author[D.~Sagaki]{Daisuke Sagaki}
\address[Daisuke Sagaki]{Department of Mathematics, 
Faculty of Pure and Applied Sciences, University of Tsukuba, 
1-1-1 Tennodai, Tsukuba, Ibaraki 305-8571, Japan}
\email{sagaki@math.tsukuba.ac.jp}

\begin{abstract} In this paper we derive a counterpart of the well-known Ram-Yip formula for symmetric and nonsymmetric Macdonald polynomials of arbitrary type. Our new formula is in terms of a generalization of the Lakshmibai-Seshadri paths (originating in standard monomial theory), which we call pseudo-quantum Lakshmibai-Seshadri (LS) paths. This model carries less information than the alcove walks in the Ram-Yip formula, and it is therefore more efficient. Furthermore, we construct a connected pseudo-crystal structure on the pseudo-quantum LS paths, which is expected to lead to a simple Littlewood-Richardson rule for multiplying Macdonald polynomials. By contrast with the Kashiwara crystals, our pseudo-crystals have edges labeled by arbitrary roots. 
\end{abstract}

\maketitle

\section{Introduction}

Macdonald polynomials are an important family of polynomials associated with an irreducible affine root system, with rational function coefficients in $q,t$. They have deep connections with several areas of mathematics: double affine Hecke algebras (DAHA), the representations of affine Lie algebras, $p$-adic groups, integrable systems,  conformal field theory,  statistical mechanics,  Hilbert schemes, etc. 

There are two versions of Macdonald polynomials: the symmetric ones (under the Weyl group action), denoted $P_\lambda(q,t)$, where $\lambda$ is a nonaffine dominant weight, and the nonsymmetric ones $E_\mu(q,t)$, where $\mu$ is an arbitrary nonaffine weight. They were defined in~\cite{M} in the setup of the corresponding double affine Hecke algebras, and are orthogonal with respect to the so-called Macdonald scalar product. 

The symmetric and nonsymmetric Macdonald polynomials generalize the irreducible characters and Demazure characters of simple Lie algebras, respectively, which are recovered by setting $q=t=0$. Other specializations are also important: $P_\lambda(q=0,t)$ are the Hall-Littlewood polynomials; $P_\lambda(q,t=q^\gamma)$, with $\gamma$ a positive real number, become the Jack polynomials by setting $q\to 1$; $E_\mu(q=\infty,t)$ are the Iwahori Whittaker functions in number theory; $E_\mu(q,t=0)$ and $E_\mu(q,t=\infty)$, as well as their symmetric versions, are closely related to the representation theory of quantum affine algebras (see below).  

Given their prominent role, it is important to have explicit formulas for the monomial expansions of Macdonald polynomials, which turn out to have a combinatorial nature. The first such formulas were given in type $A$: Macdonald's formula for $P_\lambda(q,t)$ in terms of semistandard Young tableaux~\cite[Chapter VI, Section 7]{M0}, and the Haglund-Haiman-Loehr formulas for $P_\lambda(q,t)$ and $E_\mu(q,t)$ in~\cite{HHL1} and~\cite{HHL2}, respectively, which are in terms of so-called non-attacking fillings of certain diagrams. Other formulas for the type $A$ Macdonald polynomials were recently derived through connections with a statistical mechanics model called the multispecies asymmetric simple exclusion process (ASEP) on a circle~\cite{CHMMW, CMW}, as well as via integrable lattice models~\cite{GW}. On another hand, building on work of Schwer~\cite{S}, Ram and Yip~\cite{RY} gave a uniform combinatorial formula for both the symmetric and nonsymmetric Macdonald polynomials of arbitrary type, in terms of so-called alcove walks. Recently, a very interesting generalization of the Ram-Yip formula was derived in~\cite{sa} for the so-called metaplectic Macdonald polynomials in~\cite{ssv1, ssv2} (which specialize to the metaplectic Whittaker functions in number theory).

Several specializations of the Ram-Yip formula, in some of the cases mentioned above, were worked out in~\cite{Le2, OS}. In~\cite{Le1, GR} it was shown that the Haglund-Haiman-Loehr formulas can be derived as ``compressed'' versions of the Ram-Yip formulas. On another hand, several applications of the Ram-Yip formula and its specializations were derived. Among them is the uniform combinatorial model, called the quantum alcove model, for tensor products of so-called (single-column) Kirillov-Reshetikhin (KR) crystals of quantum affine algebras~\cite{2, 3}; see~\cite{kas} and below for the theory of Kashiwara crystals. A corollary of this work is the identification of $P_\lambda(q,t=0)$ and $E_\mu(q,t=0)$ with graded characters of tensor products of KR modules and certain Demazure-type submodules, respectively. The quantum alcove model extends the alcove model in~\cite{LP}, which is a uniform combinatorial model for the characters and crystals corresponding to the irreducible integrable highest weight modules for complex symmetrizable Kac-Moody algebras. On another hand, the alcove model can be viewed as a discrete counterpart of the Littelmann path model~\cite{L1, L2}.

In this paper we derive new combinatorial formulas for the symmetric and nonsymmetric Macdonald polynomials in terms of so-called pseudo-quantum Lakshmibai-Seshadri (LS) paths; see Theorem~\ref{thm:graded_character}. We also specialize our formulas to similar ones for Jack and Hall-Littlewood  polynomials; see Corollaries~\ref{jack} and~\ref{hl}, respectively. LS paths are certain piecewise-linear paths which were originally defined in the context of standard monomial theory~\cite{LS1}. They are special cases of Littelmann paths~\cite{L1, L2}. The generalization of LS paths called quantum LS paths was used in~\cite{2, 3} as yet another model for tensor products of (single-column) KR crystals. Here we generalize the quantum LS paths even further, by considering the mentioned pseudo-quantum LS paths. Essentially, the mentioned generalizations are obtained by replacing a certain condition based on the Bruhat order on the Weyl group (for LS paths) with similar conditions based on the so-called quantum Bruhat graph (for quantum LS paths, see~\cite{1}) and the double Bruhat graph, considered here (for pseudo-quantum LS paths).

We derive the new formulas for the symmetric and nonsymmetric Macdonald polynomials by translating the corresponding Ram-Yip formulas, namely by bijecting the respective alcove walks and pseudo-quantum LS paths, and by translating the corresponding statistics. A similar bijection was constructed in~\cite{2} in the special case of the quantum alcove model and quantum LS paths, and earlier in~\cite{LP} for the alcove model and LS paths. The advantage of LS path-type models is that they carry less information than models based on alcove walks, and therefore are more efficient. Note that the quantum LS paths were successfully used in studying connections between specializations of Macdonald polynomials and the representation theory of quantum affine algebras (in particular, the so-called level-zero extremal weight modules); this was carried out in several recent papers, namely~\cite{NNS, nnstpd, NS, nasdsl, naslzv, nomqls, nomgwm}. Therefore, we expect similar applications of our pseudo-quantum LS paths.

Finally, in Section~\ref{pseudo-cryst} we construct a pseudo-crystal structure on pseudo-quantum LS paths. Crystals were defined by Kashiwara~\cite{kas} as (connected) colored directed graphs encoding certain representations of quantum affine algebras (including the irreducible highest weight ones) in the limit of the quantum parameter going to zero; the edge colors correspond to labeling the edges by simple roots. By contrast, our pseudo-crystals have edges labeled by arbitrary roots, and are also shown to be connected graphs. Therefore, for a fixed dominant weight $\lambda$, the pseudo-crystal structure can be used
as an alternative way to generate all the pseudo-quantum LS paths of shape $\lambda$ 
by starting from the special one $(e; 0, 1)$, which is thought of as the straight-line path joining the origin to $\lambda$.

Crystals were successfully used to solve several basic problems in representation theory, such as decomposing tensor products of irreducible representations. Essentially, such a Littlewood-Richardson rule is easily derived by considering the highest weight vertices in the connected components of a tensor product of the corresponding crystals, which is defined via a specific tensor product rule. The Littlewood-Richardson rule can be made very explicit by using Littelmann paths~\cite{L1, L2} or the alcove model~\cite{LP} for the mentioned crystals, as pointed out above.

By analogy, we expect our pseudo-crystal structure on pseudo-quantum LS paths to have applications to a Littlewood-Richardson rule for Macdonald polynomials; this would express the product of two Macdonald polynomials in terms of Macdonald polynomials. Such a rule was derived by different methods in~\cite{yip}. However, it is quite involved, while we expect the proposed approach to lead to a simpler rule. Indeed, the Littlewood-Richardson rule for the irreducible representations based on tensor products of crystals is the simplest such rule. On another hand, it would be interesting to extend this approach to the metaplectic Macdonald polynomials mentioned above.

\subsection*{Acknowledgments}

 C.L. was partially supported by the NSF grant DMS-1855592. 
 S.N. was partially supported by JSPS Grant-in-Aid for Scientific Research (C) 21K03198. 
 D.S. was partially supported by JSPS Grant-in-Aid for Scientific Research (C) 19K03415. 

\section{Macdonald polynomials in terms of pseudo-quantum LS paths}

\subsection{Pseudo-quantum Lakshmibai-Seshadri paths}

Let $\mathfrak{g}$ be a finite-dimensional simple Lie algebra over $\mathbb{C}$,
$I$ the vertex set for the Dynkin diagram of  $\mathfrak{g}$,
and
$\{ \alpha_i \}_{i \in I }$
(resp., $\{ {\alpha}^{\lor}_i \}_{i \in I }$)
 the set of simple roots (resp., coroots) of  $\mathfrak{g}$.
Then
$\mathfrak{h} = \bigoplus_{i \in I}\mathbb{C}\alpha^\lor_i$ is a Cartan subalgebra of  $\mathfrak{g}$,
with
$\mathfrak{h}^* = \bigoplus_{i \in I}\mathbb{C}\alpha_i$ the dual space of $\mathfrak{h}$ 
and $\mathfrak{h}_\mathbb{R}^* = \bigoplus_{i \in I}\mathbb{R}\alpha_i$ its real form;
the canonical pairing between  $\mathfrak{h}$ and  $\mathfrak{h}^*$ is denoted by
$\langle \cdot, \cdot \rangle : \mathfrak{h}^* \times \mathfrak{h} \rightarrow \mathbb{C}$.
Let $Q = \sum_{i \in I}\mathbb{Z}\alpha_i  
\subset \mathfrak{h}_\mathbb{R}^*$ denote 
the root lattice,
$Q^\lor = \sum_{i \in I}\mathbb{Z}\alpha_i^\lor  
\subset \mathfrak{h}_\mathbb{R}$ 
the coroot lattice,
and 
$P = \sum_{i \in I}\mathbb{Z}\varpi_i \subset \mathfrak{h}_\mathbb{R}^*$  the weight lattice of $\mathfrak{g}$,
where the $\varpi_i$, $i \in I$, are the fundamental weights for $\mathfrak{g}$,
i.e., 
$\langle \varpi_i ,\alpha_j^\lor \rangle = \delta_{i j}$
for $i, j \in I$;
we set $P^+ \eqdef \sum_{i \in I} \BZ_{\geq 0} \varpi_i$, and call an element $\lambda$ of $P^+$ a dominant (integral) weight.
Let us denote by $\Delta$ the set of roots,
and by $\Delta^{+}$ (resp., $\Delta^{-}$) the set of  positive (resp., negative) roots.
Also, let $W \eqdef \langle s_i \ | \ i \in I \rangle$
be the Weyl group of $\mathfrak{g}$,
where
$s_i $, $i \in I$, are the simple reflections acting on $\mathfrak{h}^*$ and on $\mathfrak{h}$
as follows:
\begin{align*}
s_i \nu = \nu - \langle  \nu , \alpha^\lor_i  \rangle \alpha_i & \ \ \
\text{for } \nu \in \mathfrak{h}^*,\\
s_i h = h - \langle \alpha_i , h \rangle \alpha^\lor_i
& \ \ \
\text{for } h \in \mathfrak{h};
\end{align*}
we denote the identity element and the longest element of $W$ by $e$ and $\lon$, respectively.
If $\alpha \in \Delta$ is written as $\alpha = w \alpha_i$ for 
$w\in W$ and $i \in I$, 
then its coroot $\alpha^\lor$ is $w \alpha^\lor_i$;
we often identify $s_\alpha$ with $s_{\alpha^\lor}$.
For  $u \in W$,
the length of $u$ is denoted by $\ell(u)$,
 which agrees with
 the cardinality of the set
$\Delta^+ \cap u^{-1}\Delta^-$.

\begin{defi}\label{DBG}
The double Bruhat graph, denoted by $\DBG$, is the directed graph with vertex set $W$ and  directed edges 
 labeled by all positive roots;
for $u,v \in W$, and $\beta \in \Delta^+$, 
$u \xrightarrow{\beta} u s_\beta $ is an edge of $\DBG$.
An edge satisfying $\ell(u s_\beta) > \ell(u)$ (resp., $\ell(u s_\beta) < \ell(u)$)
is called a Bruhat (resp., quantum) edge.
\end{defi}

For a subset $S \subset I$,
we set
$W_S \eqdef \langle s_i \ | \ i \in S \rangle$.
We denote the longest element of $W_S$ by $\lons$.
Also, we set $\Delta_S \eqdef Q_S \cap \Delta^+$,
$\Delta_S^+ \eqdef \Delta_S \cap \Delta^+ $, and
$\Delta_S^- \eqdef \Delta_S \cap \Delta^- $,
where $Q_S \eqdef \sum_{i \in S} \mathbb{Z}\alpha_i$. 
Let $W^S (\subset W^S)$ denote the set of all minimal-length coset representatives for the cosets in $W / W_S$.
For $w\in W$, we denote  the minimal-length coset representative of the coset $w W_S$ by 
$\lfloor w \rfloor$, and
for a subset $T \subset W$, 
we set $\lfloor T \rfloor \eqdef \{ \lfloor w \rfloor \ | \ w \in T \} \subset W^S$.

%

We take and fix an arbitrary dominant weight
$\lambda \in P^+$, i.e., 
$\langle \lambda , \alpha^{\lor}_i \rangle \geq 0$
for all $i \in I$.
We set 
\begin{equation*}S = S_\lambda \eqdef \{ i \in I \ | \  \langle \lambda , \alpha^{\lor}_i \rangle =0 \} \subset I.
\end{equation*}

\begin{defi}
Let $\lambda \in P^+$ be a dominant weight and 
$b \in \mathbb{Q}\cap [0,1]$.
We denote by
$\DBG_{b\lambda}$ 
 the subgraph of $\DBG$ 
with the same vertex set but having only the edges:
$u \xrightarrow{\beta} v$ with $b\langle \lambda, \beta^{\lor}  \rangle \in \mathbb{Z}$.
\end{defi}

\begin{rem}\label{rem:1.3}
Let $\lambda \in P^+$ be a dominant weight and 
$b \in \mathbb{Q}\cap [0,1]$.
Let $u,v \in W$.
If there exists a directed path 
\begin{equation*}
u = x_0 \xrightarrow{\gamma_1}
x_1\xrightarrow{\gamma_2}
\cdots
\xrightarrow{\gamma_r}
x_r
=
v
\end{equation*}
from
$u$ 
to
$v$ 
in $\DBG_{b \lambda}$,
then
$b(v \lambda -u \lambda)\in P$
since
$b(v \lambda -u \lambda)\in 
\sum_{k=1}^{r}
\BZ \pair{\lambda}{\gamma_k^\lor}$.
\end{rem}

For $\lambda \in P^+$ and $b \in \mathbb{Q} \cap [0,1]$, 
we define the $(b, \lambda)$-degree of an edge $u \xrightarrow{\beta} v $ in $\DBG_{b \lambda}$ (resp.,  $\DBG^S_{b \lambda}$) by:
\begin{equation*}
\dg_{b\lambda} (u \rightarrow v)
=
\left\{
\begin{array}{ll}
      \frac{ t^{-\half}(1-t) }{1- q^{(1-b)\langle \lambda, \beta^{\lor}  \rangle}t^{\langle \rho , \beta^{\lor }\rangle}}
 & \mbox{if} \ u \xrightarrow{\beta} v \mbox{ is a Bruhat edge}, \\
      \frac{ q^{(1-b)\langle \lambda, \beta^{\lor}  \rangle}t^{- \half + \langle \rho , \beta^{\lor }\rangle}(1-t)}{1- q^{(1-b)\langle \lambda, \beta^{\lor}  \rangle}t^{\langle \rho , \beta^{\lor }\rangle}} &  \mbox{if} \ u \xrightarrow{\beta} v \mbox{ is a quantum edge},
\end{array}
\right.
\end{equation*}
where $q,t$ are indeterminates, and $\rho \eqdef \frac{1}{2}\sum_{\alpha \in \Delta^+} \alpha$.

\begin{defi}\label{def_qls}
Let $\lambda \in P^+$,
and set $S = S_\lambda = \{ i \in I \ | \ \langle \lambda , \alpha^\lor_i \rangle =0 \}$.
A pair $\eta = (w_1, w_2 ,\ldots ,w_s ; \tau_0, \tau_1 , \ldots , \tau_s ) $
of a sequence
$w_1,\ldots , w_s$ of  elements in $W^S$  and a increasing sequence  
$0=\tau_0<\cdots < \tau_s=1$ of rational numbers,
is called a pseudo-quantum Lakshmibai-Seshadri ($\PQLS$) path of shape $\lambda$
if 

(C)
for every $1\leq i \leq s-1$, there exists a directed  path from $w_{i+1}$ to $w_{i}$ in $\DBG_{\tau_i \lambda}$.

\noindent
Let $\BPQLS(\lambda)$ denote the set of all $\PQLS$ paths of shape $\lambda$.
\end{defi}

We set 
$\BPQLS^\mu(\lambda)
\eqdef \{ (w_1
, w_2, \ldots , w_s; \sigma_0, \ldots ,\sigma_s) \in \BPQLS(\lambda) \ | \
w_1 \leq v(\mu)\}$, $\mu \in W\lambda$,
where $\leq$ denotes the Bruhat order on $W$.
%
%
%
%
For $\eta=(w_1, w_2 ,\ldots ,w_s ; \tau_0, \tau_1 , \ldots , \tau_s ) \in \BPQLS(\lambda)$,
we set 
\begin{equation*}
\wt(\eta) 
\eqdef
 \sum^{s-1}_{i=0}(\tau_{i+1}-\tau_i) w_{i+1} \lambda.
\end{equation*}

\subsection{Relationship between $\BPQLS^{\mu}(\lambda)$ and $\PQLS^{\mu}(\lambda)$.}

Let $\lambda \in P^+$ be a dominant weight, 
and $\mu \in W\lambda$.
We set  $\lambda_- \eqdef \lon \lambda$.
We denote by $v(\mu) \in W^S$
the minimal-length coset representative for the coset
$\{ w \in W \ | \ w \lambda =\mu \}$
in $W/ W_S$;
recall that 
$S = S_\lambda = \{ i\in I \ | \ \langle \lambda, \alpha^\lor_i \rangle = 0 \}$.
Then
we have $\ell(v(\mu)w)=\ell(v(\mu))+ \ell(w)$ for all $w \in W_S$.
In particular, we have $\ell(v(\mu)\lons )=\ell(v(\mu))+ \ell(\lons)$.
When $\mu=\lambda_-$, it is clear that
$\lon \in \{ w \in W \ | \ w\lambda = \lambda_- \}$.
Since $\lon$ is the longest element of $W$, we have $\lon = v(\lambda_-) \lons$
and $\ell(v(\lambda_-)\lons )=\ell(v(\lambda_-))+ \ell(\lons)$;
note that $v(\lambda_-) = \lon \lons = \lfloor \lon \rfloor$.

We fix reduced expressions
\begin{align*}
v(\lambda_-) v(\mu)^{-1}&=  s_{i_1}\cdots s_{i_{K}} , \\ 
v(\mu) &= s_{i_{K+1}}\cdots s_{i_{M}},\\
\lons &= s_{i_{M+1}}\cdots s_{i_N},
\end{align*}
for 
$v(\lambda_-)v(\mu)^{-1}$, $v(\mu)$, and $\lons$, respectively.
Then, 
$v(\lambda_-) =  s_{i_1}\cdots s_{i_M}$
and
$\lon =  s_{i_1}\cdots s_{i_N}$
are reduced expressions for $v(\lambda_-)$ and $\lon$, respectively.
We set
$\beta_j \eqdef s_{i_N}\cdots s_{i_{j+1}}\alpha_{i_j}$ for $1 \leq j \leq N$.
Then we have
\begin{equation}\label{eq:beta}
\begin{split}
\Delta^+ \setminus \Delta_S^+ &= \{ \beta_{1}, \ldots , \beta_{M}\},\\
\Delta^+_S &= \{ \beta_{M+1}, \ldots , \beta_{N}\},\\
\Delta^+ &= \{ \beta_{1}, \ldots , \beta_{N}\},\\
\lons (\Delta^+ \cap v(\mu)\Delta^-) &= \{ \beta_{K+1}, \ldots , \beta_{M}\}. \\
\end{split}
\end{equation}

We define a total order $\succ$ on $\Delta^+$ by:
\begin{align}\label{weak_reflection_order}
\underbrace{ \beta_1 \succ \beta_2 \succ \cdots \succ \beta_M }_{\in  \Delta^+ \setminus \Delta_S^+} \succ 
\underbrace{ \beta_{M+1} \succ \beta_{N}}_{\in \Delta_S^+}
.
\end{align}

\begin{rem}
The total order $\succ$ above is a  reflection order on $\Delta^+$; that is,
if $\alpha ,\beta ,\gamma \in \Delta^+$ with $\gamma^\lor = \alpha^\lor + \beta^\lor$,
then $\alpha \prec \gamma \prec \beta$ or $\beta \prec \gamma \prec \alpha$.
\end{rem}

For
$\alpha \in \Delta$,
we set
\begin{align*}
|\alpha|
\eqdef
\left\{
\begin{array}{ll}
				\alpha  & \ \ \mbox{if} \ \alpha \in \Delta^+  ,  \\
				-\alpha  & \ \ \mbox{if} \ \alpha \in \Delta^-.  
			\end{array}
\right.
\end{align*}
\begin{lemm}\label{involution}
\mbox{}%
\begin{enu}
\item
Let $w \in W$.
If $x \xrightarrow{\beta} y$ is an edge of $\DBG$, then 
$wx \xrightarrow{\beta} wy$ and
$yw \xrightarrow{| w^{-1} \beta |} xw$ are also edges of $\DBG$.

\item
If $x \xrightarrow{\beta} y$ is a Bruhat $($resp., quantum$)$ edge of $\DBG$, then $y\lon \xrightarrow{- \lon \beta} x \lon$ is also a Bruhat $($resp., quantum$)$ edge of $\DBG$.
\end{enu}
\end{lemm}

\begin{proof}
Part (1) is obvious.
Part (2) follows easily from the equality
$\ell(y)-\ell(x)=\ell(x \lon )-\ell(y \lon )$. 
This proves the lemma.
\end{proof}

\begin{lemm}\label{parabolic_DBG}
Let $\lambda \in P^+$ be a dominant weight,
and let $b \in \mathbb{Q}\cap [0,1]$.
Let $w,v \in W$,
and assume that there exists a directed path
from $w$ to $v$ in $\DBG_{b \lambda}${\rm;}
let
$q \in \mathbb{Z}_{\geq 0}$ be the length of this path.
Then,
the following hold.
\begin{enu}
\item
There exists a directed path
\begin{equation*}
w=x_0 \xrightarrow{\gamma_1} x_1 \xrightarrow{\gamma_2} \cdots \xrightarrow{\gamma_p} x_p =v
\end{equation*}
in $\DBG_{b \lambda}$
such that $p \leq q$ and $\gamma_1 \prec \cdots \prec \gamma_p$.
Moreover, there exists an element $w' \in w W_S$
such that
there exists a label-increasing directed path from $w'$ to $v$ in $\DBG_{b \lambda}$ 
whose edge labels lie in $\Delta^+ \setminus \Delta^+_S$.

\item
For each $w'' \in wW_S$ and $v'' \in vW_S$,
there exists a directed path from $w''$ to $v''$ in $\DBG_{b \lambda}$.
In particular,
there exists a directed path from $\lfloor w \rfloor$ to $\lfloor v \rfloor$ in $\DBG_{b \lambda}$.
\end{enu}
\end{lemm}

\begin{proof}
Let
\begin{equation}\label{path_2.1}
w=y_0 \xrightarrow{\xi_1} y_1 \xrightarrow{\xi_2} \cdots \xrightarrow{\xi_q} y_q =v
\end{equation}
be a directed path in $\DBG_{b \lambda}$.

(1)
We will prove the assertion of the lemma by induction on $q$. If $q=1$, then the assertion is obvious.
Assume that $q>1$.
We write $\xi_1 = \beta_j$ for some $1 \leq j \leq N$;
we prove the assertion (for the $q$) by descending induction on $j$.
By the induction hypothesis on $q$,
there exists a directed path
\begin{equation*}
y_1 = z_0 \xrightarrow{\zeta_1} z_1 \xrightarrow{\zeta_2} \cdots \xrightarrow{\zeta_r} z_r =v
\end{equation*}
from $y_1$ to $v$ in $\DBG_{b \lambda}$ such that $r \leq q-1$ and $\zeta_1 \prec \cdots \prec \zeta_r$.
Thus we obtain
\begin{equation}\label{eq:1.2}
w=y_0 \xrightarrow{\xi_1} y_1 = z_0 \xrightarrow{\zeta_1} z_1 \xrightarrow{\zeta_2} \cdots \xrightarrow{\zeta_r} z_r =v
\end{equation}
If $r<q-1$, then we can apply the induction hypothesis on $q$ to the directed path above.
Hence we may assume that $r=q-1$.
Assume that $j=N$, i.e., $\xi_1 = \beta_N$.
Then we have $\xi_1 \preceq \zeta_1$ since $\xi_1 = \beta_N$ is the minimal element of $\Delta^+$.
If $\xi_1 \prec \zeta_1$, then the directed path (\ref{eq:1.2}) has the desired property.
If $\xi_1 = \zeta_1$, then we see that $ w = z_1 $, and that the path
\begin{equation*}
w = z_1 \xrightarrow{\zeta_2} \cdots \xrightarrow{\zeta_r} z_r =v
\end{equation*}
has the desired property.
Assume now that $j<N$.
If $\xi_1 \preceq \zeta_1$, then we obtain a directed path having the desired property in exactly the same way as for $j=N$.
Assume that 
$\xi_1 \succ \zeta_1$.
Then we see that 
\begin{equation}\label{eq:1.3}
w=y_0 \xrightarrow{\zeta_1} w s_{\zeta_1} \xrightarrow{|s_{\zeta_1} \xi_1|} z_1 \xrightarrow{\zeta_2} \cdots \xrightarrow{\zeta_r} z_r =v
\end{equation}
is a directed path of length $r+1 =q$ from $w$ to $v$ in $\DBG_{b \lambda}$;
here we remark that
$b\pair{\lambda}{s_{\zeta_1} \xi_1^\lor}\in \BZ$ 
since $b\pair{\lambda}{\xi_1^\lor}\in \BZ$
and $b\pair{\lambda}{\zeta_1^\lor}\in \BZ$.
Since $\zeta_1  \prec \xi_1$, we can apply the induction hypothesis on $j$ to this path.
Thus, we obtain a directed path
\begin{equation*}
w=x_0 \xrightarrow{\gamma_1} x_1 \xrightarrow{\gamma_2} \cdots \xrightarrow{\gamma_p} x_p =v
\end{equation*}
in $\DBG_{b \lambda}$
such that $p \leq q$ and $\gamma_1 \prec \cdots \prec \gamma_p$.

Also, by (\ref{weak_reflection_order}),
there exists $1 \leq k \leq p$ such that 
$\{ \gamma_1 , \ldots, \gamma_k \} \subset \Delta_S^+$
and
$\{ \gamma_{k+1} , \ldots, \gamma_p \} \subset  \Delta^+ \setminus \Delta_S^+$.
Then the directed path 
\begin{equation*}
x_k \xrightarrow{\gamma_{k+1}} x_1 \xrightarrow{\gamma_2} \cdots \xrightarrow{\gamma_p} x_p =v
\end{equation*}
is a path from $x_k$ to $v$ in $\DBG_{b \lambda}$ whose edge labels lie in $\Delta^+ \setminus \Delta^+_S$.
In addition, 
$x_k \in w W_S$ since $\{ \gamma_1 , \ldots, \gamma_k \} \subset \Delta_S^+$.
This proves part (1) of the lemma.

(2)
Let 
\begin{align}
w'' = y_{-s} \xrightarrow{\xi_{-s+1}}  \cdots \xrightarrow{\xi_0} y_0 =w, \label{path_2.4} \\
v= y_k \xrightarrow{\xi_{q+1}}  \cdots \xrightarrow{\xi_t} y_t =v'' \label{path_2.5}
\end{align}
be a directed path from 
$w''$ to $w$ and one from $v$ to $v''$, respectively, whose edge labels lie in $\Delta^+_S$.
Since $\langle \lambda, \xi_{k}^\lor \rangle =0$ for $ -s+1 \leq k \leq 0$ or $q+1 \leq k \leq t$,
the directed paths (\ref{path_2.4}) and (\ref{path_2.5}) are paths in $\DBG_{b \lambda}$.
Therefore, the path given by concatenating the paths (\ref{path_2.1}), (\ref{path_2.4}), and (\ref{path_2.5})
is a directed path from $w''$ to $v''$ in $\DBG_{b \lambda}$.
This proves part (2) of the lemma.
\end{proof}

\begin{defi}\label{defi:pQLS}
Let $\lambda \in P^+$,
and set $S = S_\lambda = \{ i \in I \ | \ \langle \lambda , \alpha^\lor_i \rangle =0 \}$.
Let
\begin{equation*}
\widetilde{\eta}=
 \left(
\left( 
w_{p, 0} \xleftarrow{\beta_{p,1}} \cdots \xleftarrow{\beta_{p,t_p}}w_{p, t_p}
\right)_{p=0, \ldots , s-1};
\sigma_0, \ldots , \sigma_s \right)
\end{equation*}
be a pair of label-increasing directed paths in $\DBG$ whose edge labels lie in $\Delta^+ \setminus \Delta^+_S$
and an increasing sequence $0 = \sigma_0 < \cdots < \sigma_s =1$ of rational numbers.
The pair $\widetilde{\eta}$ is called {\it a pQLS pair of shape $\lambda$} if

(i)
$\sigma_p \langle \lambda , \beta^\lor_{p,m} \rangle \in \mathbb{Z}$ for $0 \leq p \leq s$, $1\leq m \leq t_p$,

(ii)
$w_{p, t_p} = w_{p+1,0}$ for  $0\leq p \leq s-2$,

(iii)
$\lons \beta_{0, m} \in \Delta^+ \cap v(\mu)^{-1} \Delta^-$ for $1 \leq m \leq t_0$,

(iv)
$w_{0,0}= v(\mu) \lons$;

\noindent
Recall from \cite[(2.2.4) and (2.7.3)]{M} that
$  \Delta^+ \cap v(\mu)^{-1} \Delta^- \subset  \Delta^+\setminus \Delta^+_S = \lons ( \Delta^+\setminus \Delta^+_S )$.

Let $\PQLS^\mu(\lambda)$ denote the set of all pQLS pairs of shape $\lambda$.
Also, for $w \in W$, we set
\begin{equation*}
w\widetilde{\eta}\eqdef
 \left(
\left( 
w w_{p, 0} \xleftarrow{\beta_{p,1}} \cdots \xleftarrow{\beta_{p,t_p}}w w_{p, t_p}
\right)_{p=0, \ldots , s-1};
\sigma_0, \ldots , \sigma_s \right),
\end{equation*}
and $w(\PQLS^\mu(\lambda)) \eqdef \{ w \widetilde{\eta} \ | \ \widetilde{\eta} \in \PQLS^\mu(\lambda) \}$;
remark that for $\widetilde{\eta} \in \PQLS^\mu(\lambda)$, $w \widetilde{\eta}$ satisfies conditions (i)-(iii).
\end{defi}

\begin{pro}\label{fiber}
We define a map
$\pr : w(\PQLS^\mu (\lambda)) \rightarrow \BPQLS(\lambda)$
by
\begin{align*}
\pr &\left(
 \left(
\left( 
w_{p, 0} \xleftarrow{\beta_{p,1}} \cdots \xleftarrow{\beta_{p,t_p}}w_{p, t_p}
\right)_{p=0, \ldots , s-1};
\sigma_0, \ldots , \sigma_s \right)
\right) \\
&\eqdef
(\lfloor w_{0, t_0}\rfloor, \ldots , \lfloor  w_{s-1, t_{s-1}} \rfloor ; \sigma_0, \ldots , \sigma_{s-1}, 1).
\end{align*}
Then the following hold.

\begin{enu}
\item
$\pr \left(\PQLS^\mu (\lambda) \right) = \BPQLS^{\mu}(\lambda)$. 

\item
If
$\mu = \lambda$, then $t_0=0$ and
$\pr \left(\bigsqcup_{w \in W^S} w( \PQLS^{\lambda} (\lambda) ) \right) = \BPQLS(\lambda)$.
\end{enu}
\end{pro}
\begin{proof}
(1)
We write $(\beta_{0,1},\ldots,\beta_{0,t_0})$ 
as $(\beta_{j_1}, \ldots, \beta_{j_r})$, 
with $j_1 < \cdots < j_r$. 
Since $v(\mu) = s_{i_{K+1}}\cdots s_{i_M}$ is a reduced expression, and $\beta_j = s_{i_N}\cdots s_{i_{j+1}}\alpha_{i_j}$ for $1 \leq j \leq N$,
we deduce that
$w_{0,t_0}=v(\mu)\lons s_{\beta_{0,1}}\cdots s_{\beta_{0,t_0}}=v(\mu)\lons s_{\beta_{j_1}}\cdots s_{\beta_{j_r}}$
is less than $v(\mu)\lons$ in the Bruhat order. Hence we have $\lfloor w_{0, t_0}\rfloor \leq v(\mu)$. Therefore, it follows that $\pr \left(\PQLS^\mu (\lambda) \right) \subset \BPQLS^{\mu}(\lambda)$.

Now, let us take an arbitrary element
$\eta = (w_1, w_2, \ldots , w_s ; \sigma_0, \ldots , \sigma_s )$ in $\BPQLS^\mu(\lambda)$, 
and set $\eta_q \eqdef (w_1, \ldots, w_q;\sigma_0,\ldots, \sigma_{q-1},1)$ for $1 \leq q\leq s$.
We will prove that $\pr \left(\PQLS^\mu (\lambda) \right) \supset \BPQLS^{\mu}(\lambda)$ by induction on $s$.
It follows from the definition of $\BPQLS^\mu(\lambda)$ that $w_1 \leq v(\mu)$.
From this, $w_1$ is obtained as a subexpression of the 
reduced expression $v(\mu) = s_{i_{K+1}}\cdots s_{i_M}$ for $v(\mu)$,
that is, 
there exists a set $J' \eqdef \{ k_1 < \cdots< k_q \} \subset \{ K+1 , \ldots , M \}$ such that $w_1 = s_{i_{k_1}}\cdots s_{i_{k_q}}$.
We set $J =\{ j_1 < \cdots < j_r \}\eqdef \{K+1, \ldots , M \} \setminus J'$.

\begin{claim}
The directed path
\begin{equation*}
v(\mu)\lons
\xleftarrow{\beta_{j_1}}
\cdots
\xleftarrow{\beta_{j_r}}
w_1 \lons
\end{equation*}
is a label-increasing path in $\DBG$, whose labels lie in $\lons (\Delta^+ \cap v(\mu)^{-1} \Delta^-)$.
\end{claim}

\noindent
$Proof \ of \ the \ claim.$
By \eqref{eq:beta},
we have $\beta_{j_k} \in \lons (\Delta^+ \cap v(\mu)^{-1} \Delta^-)$ for $1 \leq k \leq r$, since $K+1 \leq j_k \leq M$.
Also, by the definition of the total order $\prec$ on $\Delta^+$, 
we have $\beta_{j_{k}}\succ \beta_{j_{k+1}}$ for $1 \leq k \leq r-1$, since $j_k < j_{k+1}$.
This proves the claim.
\bqed

From the claim above, we see that the pair $\widetilde{\eta}_1\eqdef \left((v(\mu)\lons
\xleftarrow{\beta_{j_1}}
\cdots
\xleftarrow{\beta_{j_r}}
w_1 \lons); 0,1\right)$
satisfies conditions (i)--(iv) in Definition \ref{defi:pQLS}, and hence that $\widetilde{\eta} \in \PQLS^\mu (\lambda)$ and $\pr(\widetilde{\eta}_1)=\eta_1$.

Let $ q < s $.
Assume that
there exists a pair
\begin{equation}
\widetilde{\eta}_q \eqdef \left(
\left( 
w_{p, 0} \xleftarrow{\beta_{p,1}} \cdots \xleftarrow{\beta_{p,t_p}}w_{p, t_p}
\right)_{p=0, \ldots , q-1};
\sigma_0, \ldots , \sigma_{q-1}, 1 \right)\in \PQLS^\mu (\lambda)
\end{equation}
such that $\pr(\widetilde{\eta}_q)=\eta_q$, i.e., such that
 $\lfloor w_{p,t_p}\rfloor = w_{p+1}$ for $0 \leq p \leq q-1$.
Then,
by the definition of $\BPQLS^\mu (\lambda) \subset \BPQLS(\lambda)$,
there exists a directed path from $w_{q+1}$ to $w_{q}$ in $\DBG_{\sigma_q \lambda}$.
By Lemma \ref{parabolic_DBG} (1), (2),
there exists a label-increasing directed path 
\begin{equation*}
w_{q-1, t_{q-1}} \reqdef w_{q,0 } \xleftarrow{\beta_{q, 1}}\cdots  \xleftarrow{\beta_{q, t_q}} w_{q,t_q } 
\end{equation*}
from some $w_{q, t_q} \in w_{q+1}W_S$ to $w_{q-1, t_{q-1}}$ in $\DBG_{\sigma_q \lambda}$
whose edge labels lie in $\Delta^+ \setminus \Delta^+_S$.
Thus, the element
\begin{equation}
\widetilde{\eta}_{q+1}
\eqdef
\left(
\left( 
w_{p, 0} \xleftarrow{\beta_{p,1}} \cdots \xleftarrow{\beta_{p,t_p}}w_{p, t_p}
\right)_{p=0, \ldots , q};
\sigma_0, \ldots , \sigma_{q}, 1 \right)
\end{equation}
 satisfies the condition that $\lfloor w_{p,t_p}\rfloor = w_{p+1}$ for $0 \leq p \leq q$, and also conditions (i)-(iv). 
This proves that $\pr \left(\PQLS^\mu (\lambda) \right) \supset \BPQLS^{\mu}(\lambda)$ by induction on $s$.

(2)
If $\mu= \lambda$, i.e., if $v(\mu)= \id$, then
$\Delta^+ \cap v(\mu)^{-1}\Delta^- = \emptyset$ and hence $t_0=0$ by condition (iii). 
Also, let
$\eta = (w_1, \ldots , w_s; \sigma_0,\ldots , \sigma_s) \in \BPQLS(\lambda)$,
and set
$w_1^{-1}(\eta) \eqdef (e, \ldots , \lfloor w_1^{-1}w_s \rfloor; \sigma_0,\ldots , \sigma_s) \in \BPQLS(\lambda)$.
By part (1), there exists 
$\widetilde{\eta} \in \PQLS^\lambda(\lambda)$ such that 
$\pr (\widetilde{\eta}) = w_1^{-1}\eta$.
Hence we deduce that 
$w_1 \widetilde{\eta} \in w_1 (\PQLS^\lambda(\lambda))$
and
$\pr (w_1\widetilde{\eta}) = \eta$.
This proves part (2).
\end{proof}

\subsection{Description of $E_{\mu}(q, t)$ and $P_{\lambda}(q, t)$ 
in terms of $\PQLS^{\mu}(\lambda)$.}

Let $E_{\mu}(q,t)$, $\mu \in W\lambda$, denote the nonsymmetric Macdonald polynomial, which is of the form $E_{\mu}(q,t) = e^{\mu} + \sum_{\nu < \mu}f_\nu e^{\nu} $, with 
$f_\nu \in \mathbb{Q}(q,t)$;
here the partial order $<$ on $P$ is the one in \cite[(2.7.5)]{M}.

For $\lambda \in P^+$,
we set
$m^{\lambda} \eqdef \sum_{\mu \in W \lambda}e^{\mu}$.
Let $P_{\lambda}(q,t)$ denote the symmetric Macdonald polynomial, which is of the form 
\begin{equation*}
P_{\lambda}(q,t) = m^{\lambda} + \sum_{\nu \in P^+ \atop \nu < \lambda}f_\nu m^{\nu}, \ \text{ with } f_\nu \in \mathbb{Q}(q,t);
\end{equation*}
here the partial order $<$ on $P^+$ is given by:
\begin{equation*}
\nu < \lambda \ \ardef \ \lambda - \nu \in \sum_{i \in I}\mathbb{Z}_{\geq 0} \alpha_i.
\end{equation*}

The following theorem will be proved in \S 3.

\begin{theo}\label{theorem_graded_character}\label{thm:graded_character}
For
\begin{equation*}
\widetilde{\eta} =  \left(
\left( 
w_{p, 0} \xleftarrow{\beta_{p,1}} \cdots \xleftarrow{\beta_{p,t_p}}w_{p, t_p}
\right)_{p=0, \ldots , s-1};
\sigma_0, \ldots , \sigma_s \right)
 \in w (\PQLS^\mu(\lambda)),
\end{equation*}
we set $\mathcal{R}(\widetilde{\eta})\eqdef\prod_{p=0}^{s-1} \prod_{m=1}^{t_p}
\deg_{\sigma_p \lambda}(w_{p, m-1} \xleftarrow{\beta_{p,m}} w_{p, m})$.
Then the following hold.

\begin{enu}
\item
\begin{align*}
E_\mu (q,t)
=
\sum_{\widetilde{\eta}\in \PQLS^\mu(\lambda)}
t^{\half(\ell(v(\mu) \lons) - \ell(w_{s-1, t_{s-1}}))}
e^{\wt(\pr(\widetilde{\eta}))}
\mathcal{R}(\widetilde{\eta}).
\end{align*}

\item
\begin{align*}
P_\lambda (q,t)=
\sum_{\widetilde{\eta} \in \bigsqcup_{w \in W^S} w(\PQLS^\lambda(\lambda)) }
t^{\half(\ell(w \lons )- \ell(w_{s-1, t_{s-1}}))}
e^{\wt(\pr(\widetilde{\eta}))}\mathcal{R}(\widetilde{\eta}).
\end{align*}
\end{enu}
\end{theo}

Let $\gamma$ be a positive real number.
We consider the specialization of the symmetric Macdonald polynomial $P_{\lambda}(q, t)$ at $q \to 1$ under the condition that $t = q^{\gamma}$, and denote the resulting polynomial by $J_{\lambda}^{(\gamma^{-1})}$; namely, 
we set
$J_\lambda^{(\gamma^{-1})}
:= \lim_{q \to 1} P_\lambda (q, q^{\gamma})$.
The polynomial $J_{\lambda}^{(\gamma^{-1})}$ is called a Jack polynomial.
By noting that $(1-q^{\gamma})/(1-q) \to \gamma$ as $q \to 1$, we immediately obtain the following corollary from Theorem \ref{thm:graded_character}.
\begin{cor}\label{jack}
Let $\lambda$ be a dominant weight. Then the following holds:
\begin{align*}
J_\lambda^{(\gamma^{-1})} =
\sum_{\widetilde{\eta} \in \bigsqcup_{w \in W^S} w(\PQLS^\lambda(\lambda)) }
e^{\wt(\pr(\widetilde{\eta}))}\prod_{p=0}^{s-1} \prod_{m=1}^{t_p}
\frac{1}{\gamma^{-1}\sigma_p\langle\lambda,\beta_{p,m}^\lor \rangle+
\langle\rho,\beta_{p,m}^\lor \rangle
}.
\end{align*}
\end{cor}


Also, we consider the specialization at $q = 0$ of the symmetric Macdonald polynomial $P_{\lambda}(q, t)$; the resulting polynomial $P_{\lambda}(0, t)$ is called a Hall-Littlewood polynomial.

Let
$\widetilde{\eta}=\left(
\left( 
w_{p, 0} \xleftarrow{\beta_{p,1}} \cdots \xleftarrow{\beta_{p,t_p}}w_{p, t_p}
\right)_{p=0, \ldots , s-1};
\sigma_0, \ldots , \sigma_s 
\right) \in \PQLS(\lambda)$. $\widetilde{\eta}$ is called a pseudo-LS path of shape $\lambda$
if $w_{p,m-1}\xleftarrow{\beta_{p,m}}w_{p,m}$ is a Bruhat edge for all $0\leq p \leq s-1$, $1 \leq m \leq t_p$.
Let $\mathrm{pLS}^\lambda(\lambda) (\subset \PQLS^\lambda(\lambda))$ denote the set of pseudo-LS paths of shape $\lambda$.
When $q=0$, we have 
\begin{equation*}
\dg_{b\lambda} (u \rightarrow v)
=
\left\{
\begin{array}{ll}
  t^{-\half}(1-t) 
 & \mbox{if} \ u \xrightarrow{\beta} v \mbox{ is a Bruhat edge}, \\
0 &  \mbox{if} \ u \xrightarrow{\beta} v \mbox{ is a quantum edge}.
\end{array}
\right.
\end{equation*}
Therefore, we obtain the following corollary from Theorem \ref{thm:graded_character}.
\begin{cor}\label{hl}
Let $\lambda$ be a dominant weight. Then the following holds:
\begin{align*}
P_\lambda (0,t)=
\sum_{\widetilde{\eta} \in \bigsqcup_{w \in W^S} w(\mathrm{pLS}^\lambda(\lambda))}
e^{\wt(\pr(\widetilde{\eta}))}\prod_{p=0}^{s-1} 
\left(t^{-\half}(1-t) \right)^{t_p}.
\end{align*}
\end{cor}

\begin{ex}
Let $\mathfrak{g}$ be of type $A_2$, $\lambda=\varpi_1+\varpi_2$, and $\mu= \lon \lambda$. We fix a reduced expression $\lon =s_1 s_2 s_1$ for $\lon$; 
the total order $\prec$ on $\Delta^+$ is:
$\alpha_2 \succ \alpha_1 + \alpha_2 \succ \alpha_1$.
Then, $S=S_\lambda=\emptyset$, and the elements $\widetilde{\eta}$ of $\PQLS^\mu(\lambda)$ are as follows:
\begin{center}
\begin{tabular}{|c|} \hline
$\widetilde{\eta}_1\eqdef\left((\lon);0,1\right)$ \\ \hline 
$\widetilde{\eta}_2\eqdef\left((\lon \xleftarrow{\alpha_1} s_1 s_2);0,1\right)$\\ \hline 
$\widetilde{\eta}_3\eqdef\left((\lon \xleftarrow{\alpha_2} s_2 s_1);0,1\right)$\\ \hline 
$\widetilde{\eta}_4\eqdef\left((\lon \xleftarrow{\alpha_1 + \alpha_2}  e);0,1\right)$\\ \hline
$\widetilde{\eta}_5\eqdef\left((\lon \xleftarrow{\alpha_2} s_2 s_1 \xleftarrow{\alpha_1 + \alpha_2} s_1);0,1\right)$\\ \hline 
$\widetilde{\eta}_6\eqdef\left((\lon \xleftarrow{\alpha_1 + \alpha_2} e \xleftarrow{\alpha_1} s_1);0,1\right)$\\ \hline 
$\widetilde{\eta}_7\eqdef\left((\lon \xleftarrow{\alpha_2} s_2 s_1 \xleftarrow{\alpha_1} s_2);0,1\right)$\\ \hline 
$\widetilde{\eta}_8\eqdef\left((\lon \xleftarrow{\alpha_2} s_2 s_1 \xleftarrow{\alpha_1 + \alpha_2} s_1 \xleftarrow{\alpha_1} e);0,1\right)$\\ \hline 
$\widetilde{\eta}_9\eqdef\left((\lon), (\lon \xleftarrow{\alpha_1 + \alpha_2} e);0,\frac{1}{2},1\right)$\\ \hline 
$\widetilde{\eta}_{10}\eqdef\left((\lon \xleftarrow{\alpha_1} s_1 s_2), (s_1 s_2 \xleftarrow{\alpha_1 + \alpha_2} s_2);0,\frac{1}{2},1\right)$\\ \hline 
$\widetilde{\eta}_{11}\eqdef\left((\lon \xleftarrow{\alpha_2} s_2 s_1), (s_2 s_1 \xleftarrow{\alpha_1 + \alpha_2} s_1);0,\frac{1}{2},1\right)$\\ \hline 
$\widetilde{\eta}_{12}\eqdef\left((\lon \xleftarrow{\alpha_1 + \alpha_2}  e), (e \xleftarrow{\alpha_1 + \alpha_2} \lon);0,\frac{1}{2},1\right)$\\ \hline
$\widetilde{\eta}_{13}\eqdef\left((\lon \xleftarrow{\alpha_2} s_2 s_1 \xleftarrow{\alpha_1 + \alpha_2} s_1), (s_1 \xleftarrow{\alpha_1 + \alpha_2} s_2 s_1);0,\frac{1}{2},1\right)$\\ \hline 
$\widetilde{\eta}_{14}\eqdef\left((\lon \xleftarrow{\alpha_1 + \alpha_2} e \xleftarrow{\alpha_1} s_1), (s_1 \xleftarrow{\alpha_1 + \alpha_2} s_2 s_1);0,\frac{1}{2},1\right)$\\ \hline 
$\widetilde{\eta}_{15}\eqdef\left((\lon \xleftarrow{\alpha_2} s_2 s_1 \xleftarrow{\alpha_1} s_2), (s_2 \xleftarrow{\alpha_1 + \alpha_2} s_1 s_2);0,\frac{1}{2},1\right)$\\ \hline 
$\widetilde{\eta}_{16}\eqdef\left((\lon \xleftarrow{\alpha_2} s_2 s_1 \xleftarrow{\alpha_1 + \alpha_2} s_1 \xleftarrow{\alpha_1} e), (e \xleftarrow{\alpha_1 + \alpha_2} \lon);0,\frac{1}{2},1\right)$\\ \hline 
\end{tabular}
\end{center}
Hence we have
\begin{align*}
&E_{\mu}(q,t)\\
&=e^{\lon \lambda} 
+ e^{s_1 s_2\lambda}\times t^{\frac{1}{2}} \times \frac{t^{-\frac{1}{2}}(1-t)}{1-qt}
+ e^{s_2 s_1\lambda}\times t^{\frac{1}{2}} \times \frac{t^{-\frac{1}{2}}(1-t)}{1-qt}\\
&+ e^{\lambda}\times t^{\frac{3}{2}}\times \frac{t^{-\frac{1}{2}}(1-t)}{1-q^2t^2}
+e^{s_1 \lambda}\times t \times \frac{t^{-\frac{1}{2}}(1-t)}{1-qt}\frac{t^{-\frac{1}{2}}(1-t)}{1-q^2t^2}\\
&+e^{s_1 \lambda}\times t \times \frac{t^{-\frac{1}{2}}(1-t)}{1-q^2t^2}\frac{qt^{\frac{1}{2}}(1-t)}{1-qt}
+e^{s_2 \lambda}\times t \times \frac{t^{-\frac{1}{2}}(1-t)}{1-qt}\frac{t^{-\frac{1}{2}}(1-t)}{1-qt}\\
&+e^{\lambda}\times t^{\frac{3}{2}} \times \frac{t^{-\frac{1}{2}}(1-t)}{1-qt}\frac{t^{-\frac{1}{2}}(1-t)}{1-q^2t^2}\frac{t^{-\frac{1}{2}}(1-t)}{1-qt}
+e^{0}\times t^{\frac{3}{2}} \times \frac{t^{-\frac{1}{2}}(1-t)}{1-qt^2}\\
&+e^{0}\times t \times \frac{t^{-\frac{1}{2}}(1-t)}{1-qt} \frac{t^{-\frac{1}{2}}(1-t)}{1-qt^2}
+e^{0}\times t \times \frac{t^{-\frac{1}{2}}(1-t)}{1-qt} \frac{t^{-\frac{1}{2}}(1-t)}{1-qt^2}\\
&+e^{0}\times 1 \times\frac{t^{-\frac{1}{2}}(1-t)}{1-q^2t^2}\frac{qt^{\frac{3}{2}}(1-t)}{1-qt^2}+e^{0}\times t^{\frac{1}{2}} \times \frac{t^{-\frac{1}{2}}(1-t)}{1-qt}\frac{t^{-\frac{1}{2}}(1-t)}{1-q^2t^2}\frac{qt^{\frac{3}{2}}(1-t)}{1-qt^2}\\
&+e^{0}\times t^{\frac{1}{2}} \times \frac{t^{-\frac{1}{2}}(1-t)}{1-q^2t^2}\frac{qt^{\frac{1}{2}}(1-t)}{1-qt}\frac{qt^{\frac{3}{2}}(1-t)}{1-qt^2}\\
&+e^{0}\times t^{\frac{1}{2}} \times \frac{t^{-\frac{1}{2}}(1-t)}{1-qt}\frac{t^{-\frac{1}{2}}(1-t)}{1-qt}\frac{qt^{\frac{3}{2}}(1-t)}{1-qt^2}\\
&+e^{0}\times 1 \times
\frac{t^{-\frac{1}{2}}(1-t)}{1-qt}\frac{t^{-\frac{1}{2}}(1-t)}{1-q^2t^2}\frac{t^{-\frac{1}{2}}(1-t)}{1-qt}\frac{qt^{\frac{3}{2}}(1-t)}{1-qt^2},
\end{align*}
where the $i$-th term on the right-hand side corresponds to
$\widetilde{\eta}_{i}$ in the list above for $1 \le i \le 16$.
\end{ex}

\begin{ex}
Let $\mathfrak{g}$ be of type $A_2$, $\lambda=\varpi_1+\varpi_2$.
Then, $S=S_\lambda=\emptyset$, and the elements $\widetilde{\eta}$ of $\bigsqcup_{w \in W} w\left( \PQLS^\lambda(\lambda)\right)$,
together with $\wt(\widetilde{\eta})$ and $\mathcal{R}(\widetilde{\eta})$, are as follows:
\begin{center}
\begin{tabular}{|c|c|c|} \hline
$\widetilde{\eta}$ & $\wt(\widetilde{\eta})$ & $\mathcal{R}(\widetilde{\eta})$\\ \hline 
$\widetilde{\eta}_e\eqdef ((e);0,1)$& $\lambda$ & 1 \\ \hline
$\widetilde{\eta}_{s_1}\eqdef ((s_1);0,1)$& $s_1\lambda$ & 1 \\ \hline
$\widetilde{\eta}_{s_2}\eqdef ((s_2);0,1)$& $s_2\lambda$ & 1 \\ \hline
$\widetilde{\eta}_{s_1s_2}\eqdef ((s_1 s_2);0,1)$& $s_1 s_2\lambda$ & 1 \\ \hline
$\widetilde{\eta}_{s_2s_1}\eqdef ((s_2 s_1);0,1)$& $s_2 s_1 \lambda$ & 1 \\ \hline
$\widetilde{\eta}_{\lon}\eqdef ((\lon);0,1)$& $\lon \lambda$ & 1 \\ \hline
$\widetilde{\eta}_1\eqdef ((e),(e \xleftarrow{\alpha_1 + \alpha_2} \lon);0,\frac{1}{2},1)$& $0$ & $\frac{qt^{\frac{3}{2}}(1-t)}{1-qt^2}$ \\ \hline
$\widetilde{\eta}_2\eqdef ((s_1),(s_1 \xleftarrow{\alpha_1 + \alpha_2} s_2s_1);0,\frac{1}{2},1)$& $0$ & $\frac{qt^{\frac{3}{2}}(1-t)}{1-qt^2}$ \\ \hline
$\widetilde{\eta}_3\eqdef ((s_2),(s_2 \xleftarrow{\alpha_1 + \alpha_2} s_1s_2);0,\frac{1}{2},1)$& $0$ & $\frac{qt^{\frac{3}{2}}(1-t)}{1-qt^2}$ \\ \hline
$\widetilde{\eta}_4\eqdef ((s_1 s_2),(s_1 s_2 \xleftarrow{\alpha_1 + \alpha_2} s_2);0,\frac{1}{2},1)$& $0$ & $\frac{t^{-\frac{1}{2}}(1-t)}{1-qt^2}$ \\ \hline
$\widetilde{\eta}_5\eqdef ((s_2 s_1),(s_2 s_1 \xleftarrow{\alpha_1 + \alpha_2} s_1);0,\frac{1}{2},1)$& $0$ & $\frac{t^{-\frac{1}{2}}(1-t)}{1-qt^2}$ \\ \hline
$\widetilde{\eta}_6\eqdef ((\lon),(\lon \xleftarrow{\alpha_1 + \alpha_2} e);0,\frac{1}{2},1)$& $0$ & $\frac{t^{-\frac{1}{2}}(1-t)}{1-qt^2}$ \\ \hline
\end{tabular}
\end{center}
Here 
\begin{align*}
\bigsqcup_{w \in W}w \left(\mathrm{pLS}^\lambda(\lambda) \right)=\{ \widetilde{\eta}_e,
\widetilde{\eta}_{s_1},
\widetilde{\eta}_{s_2},
\widetilde{\eta}_{s_1s_2},
\widetilde{\eta}_{s_2s_1},
\widetilde{\eta}_{\lon},
\widetilde{\eta}_4,
\widetilde{\eta}_5,
\widetilde{\eta}_6
\} .
\end{align*}
Hence we have
\begin{align*}
P_{\lambda}(q,t)&=e^{\lambda}+e^{s_1 \lambda}+e^{s_2\lambda}+e^{s_1 s_2 \lambda}+e^{s_2 s_1\lambda}+e^{\lon \lambda}\\
&+e^{0}\times t^{-\frac{3}{2}}\times \frac{qt^{\frac{3}{2}}(1-t)}{1-qt^2}
+e^{0}\times t^{-\frac{1}{2}}\times \frac{qt^{\frac{3}{2}}(1-t)}{1-qt^2}
+e^{0}\times t^{-\frac{1}{2}}\times \frac{qt^{\frac{3}{2}}(1-t)}{1-qt^2}\\
&+e^{0}\times t^{\frac{1}{2}} \times\frac{t^{-\frac{1}{2}}(1-t)}{1-qt^2}
+e^{0}\times t^{\frac{1}{2}}\times\frac{t^{-\frac{1}{2}}(1-t)}{1-qt^2}
+e^{0}\times t^{\frac{3}{2}}\times \frac{t^{-\frac{1}{2}}(1-t)}{1-qt^2},
\end{align*}
\begin{align*}
P_{\lambda}(0,t)=e^{\lambda}+e^{s_1 \lambda}+e^{s_2\lambda}+e^{s_1 s_2 \lambda}+e^{s_2 s_1\lambda}+e^{\lon \lambda}+3 \times (1-t)e^0,
\end{align*}
and
\begin{align*}
J_{\lambda}^{(\gamma^{-1})}&=e^{\lambda}+e^{s_1 \lambda}+e^{s_2\lambda}+e^{s_1 s_2 \lambda}+e^{s_2 s_1\lambda}+e^{\lon \lambda}+6 \times \frac{1}{\gamma^{-1}+2}e^0.
\end{align*}
\end{ex}

\section{Ram-Yip formulas in terms of $\PQLS^\mu(\lambda)$}

\subsection{Ram-Yip formulas}
In this subsection,
we review formulas \cite[Theorems 3.1 and 3.4]{RY} for nonsymmetric and symmetric Macdonald polynomials.

Let
$\widetilde{ \mathfrak{g}}$ denote the  finite-dimensional simple Lie algebra whose simple roots are $\{ {\alpha}^{\lor}_i \}_{i \in I }$,
and whose simple coroots are
$\{ \alpha_i \}_{i \in I }$;
the Cartan subalgebra of $\widetilde{ \mathfrak{g}}$
is $\mathfrak{h}^*$, and its dual space is $\mathfrak{h}$.
We denote the set of all roots of $\widetilde{ \mathfrak{g}}$ by $\widetilde{\Delta} = \{ \alpha^\lor \ | \ \alpha \in \Delta \}$,
and 
the set of all positive (resp., negative) roots of $\widetilde{ \mathfrak{g}}$ by
$\widetilde{\Delta}^{+}$ (resp., $\widetilde{\Delta}^{-}$).
Also, for 
a subset $S \subset I$,
we set
$\widetilde{Q}_S \eqdef \sum_{i \in S}\mathbb{Z}\alpha^\lor_i$,
$\widetilde{\Delta}_S \eqdef \widetilde{\Delta} \cap \widetilde{Q}_S$,
$\widetilde{\Delta}^+_S = \widetilde{\Delta}_S \cap \widetilde{\Delta}^{+}$,
and
$\widetilde{\Delta}^-_S = \widetilde{\Delta}_S \cap \widetilde{\Delta}^{-}$.

Let $\widetilde{ \mathfrak{g}}_{\aff}$ denote the untwisted affinization of $\widetilde{ \mathfrak{g}}$.
We denote
by $\widetilde{\Delta}_{\aff}$ 
the set of all real roots,
and 
by $\widetilde{\Delta}_{\aff}^{+}$
(resp., $\widetilde{\Delta}_{\aff}^{-}$)
the set of all positive (resp., negative) real roots.
Let $\widetilde{\delta} $ be the (primitive) null root of $\widetilde{ \mathfrak{g}}_{\aff}$.
Then we have
$\widetilde{\Delta}_{\aff}=
 \{ \alpha^\lor +a \widetilde{\delta} \ | \ \alpha \in \Delta , a \in \mathbb{Z} \}$.
We set
$\alpha^{\lor}_0 \eqdef \widetilde{\delta} - \varphi^{\lor} $,
where $\varphi \in \Delta$ denotes the highest short root, 
and set
$ I_{\aff}\eqdef I\sqcup \{0 \}$.
Then, $\{ \alpha^{\lor}_i \}_{i\in  I_{\aff}}$ is the set of all simple roots of
$\widetilde{ \mathfrak{g}}_{\aff}$.
Also, for 
$\beta \in \mathfrak{h}\oplus \mathbb{C}\widetilde{\delta}$,
we define ${\dg}(\beta) \in \mathbb{C} $ and 
$\overline{\beta} \in \mathfrak{h}$
by
$\beta = {\dg}(\beta) \widetilde{\delta} + \overline{\beta}$.

We denote the Weyl group of $\widetilde{ \mathfrak{g}}$ by $\widetilde{W}$;
since we have $\widetilde{W} \cong W $,
we identify these two groups.
For $x\in \mathfrak{h}^*$,
let $t(x)$ denote the translation in $\mathfrak{h}^*$: $ t(x) y = x+y$ for $y \in \mathfrak{h}^*$.
The affine Weyl group and the extended affine Weyl group
(of $\widetilde{\mathfrak{g}}_\aff$) 
are defined by:
$\widetilde{W}_{\aff}\eqdef t(Q) \rtimes W $ and
$\widetilde{W}_{\ext} \eqdef t(P) \rtimes W $,
respectively.
Also, we define $s_0 : \mathfrak{h}^* \rightarrow \mathfrak{h}^*$ by $x \mapsto x -( \langle x, \varphi^\lor  \rangle -1)\varphi $.
Then, $\widetilde{W}_{\aff}=\langle s_i \ | \ i \in I_{\aff}\rangle$.

The extended affine Weyl group $\widetilde{W}_{\ext}$ acts on
$\mathfrak{h}\oplus \mathbb{C}\widetilde{\delta}$
as linear transformations,
and on 
$\mathfrak{h}^*$ 
as affine transformations:
for $v\in W$, $t(x) \in t(P)$,
\begin{align*}
v t(x)( \overline{\beta}+r\widetilde{\delta} )&=v\overline{\beta}+(r-\langle x, \overline{\beta} \rangle ) \widetilde{\delta}
\
\text{ for } \overline{\beta} \in  \mathfrak{h} ,\ r \in \mathbb{C},
\\
v t(x)y&=v x+v y \
\text{ for } y \in \mathfrak{h}^* .
\end{align*}
An element $u \in \widetilde{W}_{\ext} $ can be written as $u=t({{\wt}(u)}) \dir (u)$,
where ${\wt}(u) \in P$ and $ {\dir}(u) \in W$,
according to the decomposition
$\widetilde{W}_{\ext} = t(P) \rtimes W $.
For $w \in \widetilde{W}_{\ext}$,
we denote the length of $w$ by 
$\ell (w) $,
which agrees with the cardinality of the set 
$
\widetilde{\Delta}^+_\aff
\cap
w^{-1}\widetilde{\Delta}^-_\aff
$.
Also, we set 
$\Omega 
\eqdef
\{ w \in \widetilde{W}_{\ext}  \ | \ \ell(w)=0 \}$.

For $\mu \in P$, 
we denote the shortest element in the coset	$t(\mu)W$ by $m_{\mu} \in \widetilde{W}_{\ext}$.
In what follows, we fix $\mu \in P$,
and take a reduced expression $m_{\mu} = u s_{\ell_{1}}\cdots s_{\ell_{L}}
\in \widetilde{W}_{\ext} = \Omega \ltimes \widetilde{W}_{\aff}$,
where $u \in \Omega$ and $ \ell_1 , \ldots , \ell_L \in  I_{\aff}$.
We know the following from \cite[Chap. 2]{M}.
\begin{lemm}\label{vmu}
\mbox{}%
\begin{enu}
\item
$\dir(m_\mu) = v(\mu)v(\lambda_-)^{-1} $,
and
$\ell(\dir(m_\mu)) + \ell( v(\mu) ) =\ell (v(\lambda_-))$.
Moreover, 
\begin{equation}\label{A}
m_\mu 
=
t(\mu)
v(\mu)
v(\lambda_-)^{-1}.
\end{equation}

\item
$v(\mu)v(\lambda_-)^{-1}\lon= v(\mu)\lons$.

\item
$\left( v(\lambda_-)v(\mu)^{-1} \right) m_\mu =m_{\lambda_-} $,
and $\ell(  v(\lambda_-)v(\mu)^{-1} )+ \ell ( m_\mu ) = \ell ( m_{\lambda_-} ) $.
\end{enu}
\end{lemm}

For each $J = \{ j_{1} < j_{2} < j_{3} < \cdots < j_{r} \} \subset \{1,\ldots,L\}$ and $w \in W$,
we define an alcove path $p^{\OS}_{J} =
			\left( w m_{\mu} = z^{\OS}_0, z^{\OS}_{1} , \ldots , z^{\OS}_{r} ; \beta^{\OS}_{j_1} , \ldots , \beta^{\OS}_{j_r} \right)$ as follows: 
we set
$\beta^{\OS}_{k} \eqdef s_{\ell_{L}}\cdots s_{\ell_{k+1}} \alpha^{\lor}_{\ell_{k}} \in \widetilde{\Delta}^+_{\aff}$ 
for $1 \leq k \leq L$, and set
	\begin{eqnarray*}
		z^{\OS}_{0}&=&w m_{\mu} ,\\
		z^{\OS}_{1}&=&w m_{\mu}s_{\beta^{\OS}_{j_{1}}},\\
		z^{\OS}_{2}&=&w m_{\mu}s_{\beta^{\OS}_{j_{1}}}s_{\beta^{\OS}_{j_2}},\\
				&\vdots& \\
		z^{\OS}_{r}&=&w m_{\mu}s_{\beta^{\OS}_{j_{1}}}\cdots s_{\beta^{\OS}_{j_r}}.
	\end{eqnarray*}
Following \cite[\S 2.3]{RY}, 
we set
	$\B (w;m_{\mu})
	\eqdef
	\left\{ p^{\OS}_{J} \ \left| \ J \subset \{ 1,\ldots ,L \}  \right. \right\}$
and
	$\ed (p^{\OS}_{J}) \eqdef z^{\OS}_{r}\in \widetilde{W}_{\ext}$.

	\begin{rem}[\normalfont{\cite[(2.4.7)]{M}}]
		If $j \in J$,
		then $-\overline{ {\beta}^{\OS}_{j}  }^{\lor} \in {\Delta}^{+}$.
	\end{rem}

	For $p^{\OS}_{J} \in \B({\id} ; m_{\mu})$, 
let $J^+ \subset J$ (resp., $J^-  \subset J$) denote the set of all indices $j_i \in J$ for which 
	$\dir(z_{i-1}^{\OS}) 
	\xrightarrow{-\overline{ {\beta}^{\OS}_{j_i}  }^{\lor}} 
	\dir (z_{i}^{\OS})$ is a Bruhat (resp., quantum) edge of $\DBG$.
Also, for $p^{\OS}_{J} \in \B (w;m_{\mu})$, we set
\begin{equation*}
\mathcal{L}(p^{\OS}_{J})\eqdef
 t^{-\half\#J}(1-t)^{\# J}
 \prod_{j \in J^+}\frac{1}{1- q^{\dg(\beta_j^\OS)}t^{\langle \rho, - \overline{ \beta^\OS_j }\rangle}}
\prod_{j \in J^-}\frac{q^{\dg(\beta_j^\OS)}t^{\langle \rho, - \overline{ \beta^\OS_j }\rangle}}{1- q^{\dg(\beta_j^\OS)}t^{\langle \rho, - \overline{ \beta^\OS_j }\rangle}}.
\end{equation*}


\begin{pro}[{\cite[Theorem 3.1]{RY}}]\label{os}
Let $\mu \in P$. Then, the following holds:
\begin{align*}
E_{\mu}(q,t)=
\sum_{p^{\OS}_{J} \in \B({\id} ; m_{\mu}) } 
&e^{\wt (\ed (p^{\OS}_J) )}t^{\half (\ell(\dir(\ed( p^{\OS}_{J} ))) - \ell(\dir(m_{\mu})) )}
\mathcal{L}(p^{\OS}_{J}).
\end{align*}
\end{pro}

\begin{pro}[{\cite[Theorem 3.4]{RY}}]\label{sos}
Let $\lambda \in P^+$. Then, the following holds:
\begin{align*}
P_{\lambda}(q,t)=
\sum_{w \in W^S}
\sum_{p^{\OS}_{J} \in \B(w ; m_{\lambda}) } 
&e^{\wt (\ed (p^{\OS}_J) )}t^{\half (\ell(\dir(\ed( p^{\OS}_{J} )))- \ell(w \dir(m_{\lambda})) ) }
\mathcal{L}(p^{\OS}_{J}).
\end{align*}
\end{pro}

\subsection{A specific total order on $\widetilde{\Delta}^+_{\aff} \cap m^{-1}_{\lambda_-} \widetilde{\Delta}^-_{\aff}$}
Let $\lambda \in P^+$ be a dominant weight,
and set 
$\lambda_- = \lon \lambda$
and
$S =\{ i \in I \ | \ \langle \lambda , \alpha^\lor_i \rangle =0 \}$.
We fix reduced expressions 
\begin{align*}
v(\lambda_-) &=  s_{i_1}\cdots s_{i_{M}}, \\
\lons &= s_{i_{M+1}}\cdots s_{i_N},
\end{align*}
for $v(\lambda_-)$ and $\lons$, respectively. 
We set 
$\beta_k = s_{i_N} \cdots s_{i_{k+1}} \alpha_{i_k}$ for $1 \leq k \leq N$,
and use the total order $\prec$ on $\Delta^+$ defined at the beginning of $\S 2.2$.

We define an injective map $\Phi$ by:
\begin{eqnarray*}
\Phi : \widetilde{\Delta}^+_{\aff} \cap m^{-1}_{\lambda_-} \widetilde{\Delta}^-_{\aff} 
&\rightarrow&
\mathbb{Q}_{\geq 0} \times (\Delta^+\setminus \Delta^+_S ) ,\\
\beta = \overline{\beta} + \dg(\beta) \widetilde{\delta}
&\mapsto &
\left(\frac{\langle {\lambda_-} ,  \overline{\beta} \rangle -  \dg(\beta)}{\langle {\lambda_-} ,  \overline{\beta} \rangle }  
,  \lon \overline{\beta}^\lor  
\right);
\end{eqnarray*}
note that $\langle {\lambda_-} ,  \overline{\beta} \rangle >0$, $\langle {\lambda_-} ,  \overline{\beta} \rangle -  \dg(\beta) \geq 0$,
and 
$\lon \overline{\beta}^\lor \in \Delta^+\setminus \Delta^+_S $, 
since $\langle {\lambda_-} ,  \overline{\beta} \rangle = \langle {\lambda} ,  \lon \overline{\beta} \rangle >0$ 
and we know from \cite[(2.4.7) (i)]{M}
that
\begin{equation}\label{B}
\widetilde{\Delta}^+_{\aff} \cap m^{-1}_{\lambda_-} \widetilde{\Delta}^-_{\aff} =
\{ \alpha^{\lor} + a \tilde{\delta} 
\ | \
\alpha \in \Delta^-,
0 < a \leq  \langle \lambda_- , \alpha^{\lor}
\rangle 
\}.
\end{equation}
We now consider the lexicographic order $<$ on $\mathbb{Q}_{\geq 0} \times (\Delta^+\setminus \Delta^+_S)$
induced by the usual total order on $\mathbb{Q}_{\geq 0}$
and
the reverse of
the restriction to $\Delta^+\setminus \Delta^+_S$
of the 
total order $\prec$ on $\Delta^+$;
that is, for $(a, \alpha), (b, \beta) \in \mathbb{Q}_{\geq 0} \times (\Delta^+\setminus \Delta^+_S)$,
\begin{equation*}
(a, \alpha)<(b, \beta) \mbox{ if and only if } 
a<b, \mbox{ or }
a=b \mbox{ and } \alpha \succ \beta.
\end{equation*}
Then we denote by $\prec'$ 
the total order on $\widetilde{\Delta}^+_{\aff} \cap m^{-1}_{\lambda_-} \widetilde{\Delta}^-_{\aff}$
induced by  the lexicographic order on $\mathbb{Q}_{\geq 0} \times (\Delta^+\setminus \Delta^+_S)$
through the map $\Phi$,
and write $\widetilde{\Delta}^+_{\aff} \cap m^{-1}_{\lambda_-} \widetilde{\Delta}^-_{\aff} $
as 
$\left\{\gamma_1 \prec' \cdots \prec' \gamma_L \right\}$.

\begin{pro}[\normalfont{\cite[Proposition 3.1.8]{NNS}}]\label{goodreducedexpression}
Keep the notation and setting above.
Then, there exists a reduced expression $m_{{\lambda_-}}=u s_{\ell_1}\cdots s_{\ell_L}$ for $m_{\lambda_-}$, $u \in \Omega$, $\{\ell_1, \ldots , \ell_L \} \subset I_{\aff}$,
such that
$\beta^\OS_j 
\left( =
s_{\ell_L}\cdots s_{\ell_{j+1}}\alpha^\lor_{\ell_j} \right) =\gamma_j$
for all $1 \leq j \leq L$.
\end{pro}

We set $a_k \eqdef \dg(\beta^{\OS}_k) \in \mathbb{Z}_{> 0} $;
since 
$\widetilde{\Delta}^+_\aff \cap m_{\lambda_-}^{-1} \widetilde{\Delta}^-_\aff
=
\{
\beta^\OS_1, \ldots , \beta^\OS_L
\}
$,
we see by (\ref{B}) that
$0<a_k \leq \langle {\lambda_-}, \overline{\beta^\OS_k}\rangle$.

\begin{lemm}[\normalfont{\cite[Lemma 3.1.10]{NNS}}]\label{lengthadditive}
Keep the notation and setting above.
Let us rewrite the reduced expression  $u s_{\ell_1}\cdots s_{\ell_L}$ for $m_{{\lambda_-}}$
as $s_{i'_1} \cdots s_{i'_M} u s_{\ell_{M+1}} \cdots s_{\ell_L}$ 
by setting $us_{\ell_k}u^{-1} = s_{i'_k}$ with $i'_k \in I_{\aff}$,
for $1 \leq k \leq M$. 
Then, $s_{i'_1} \cdots s_{i'_M}$ is a reduced expression for $v({\lambda_-})$,
and $u s_{\ell_{M+1}} \cdots s_{\ell_L}$ is a reduced expression for $m_\lambda$.
Moreover, $i_k = i'_k$ for $ 1 \leq k \leq M$.
\end{lemm}

%

\begin{rem}[\normalfont{\cite[Theorem 8.3]{2}}, \normalfont{\cite[proof of Lemma 3.1.10]{NNS}}]\label{2.15}\label{rem:2.15}
\mbox{}%
\begin{enu}
\item
For $1 \leq k \leq L$, we set
	\begin{equation*}
		d_k \eqdef
			 \frac{\langle \lambda_- ,  \overline{\beta^{\OS}_k }\rangle - a_k}{\langle \lambda_- ,  \overline{\beta^{\OS}_k }\rangle}.  
	\end{equation*}
Here $d_k$ is just the first component of $\Phi(\beta^{\OS}_k) \in \mathbb{Q}_{\geq 0} \times (\Delta^+ \setminus \Delta^+_S)$.
For  $1 \leq k,j \leq L$,
$\Phi(\beta^\OS_k)<\Phi(\beta^\OS_j)$ if and only if $k < j$,
and hence we have
	\begin{equation}\label{C}
		0\leq d_1 \leq \cdots \leq d_L \lneqq 1.
	\end{equation}
Also, we have
\begin{equation}\label{3.1.9}
d_i = 0  \  \Leftrightarrow \ i \leq M.
\end{equation}
\item
For $1 \leq k \leq M$, 
\begin{equation}\label{lemm:akbk}
\lon \overline{\beta^\OS_k}^\lor = \beta_k.
\end{equation}
\end{enu}
\end{rem}

\begin{lemm}[\normalfont{\cite[Lemma 3.1.12]{NNS}}]\label{remark2.11}
If $1\leq k<j \leq L$ and $d_k =d_j$,
then	$\lon  \overline{ \beta^{\OS}_k }^{\lor} \succ \lon  \overline{ \beta^{\OS}_j }^{\lor}$.
\end{lemm}

\subsection{Proof of Theorem \ref{theorem_graded_character}}

Let $\lambda \in P^+$ be a dominant weight and $\mu \in W\lambda$;
recall that 
$S = \{ i \in I \ | \ \langle \lambda , \alpha^\lor_i \rangle =0 \}$ and
$\lambda_- = \lon \lambda$.

We fix reduced expressions
for $v(\lambda_-)v(\mu)^{-1}$, $v(\mu)$, and $\lons$, and the total order $\prec$ on $\Delta^+$ as at the beginning of \S 2.2.
In addition, we use the total order $\prec'$ on $\widetilde{\Delta}^+_\aff \cap m_{\lambda_-}^{-1}
\widetilde{\Delta}^-_\aff$ defined just before Proposition
\ref{goodreducedexpression},
and take the reduced expression 
$m_{\lambda_-} = u s_{\ell_1}\cdots s_{\ell_L}$
for $m_{\lambda_-}$ given by Proposition \ref{goodreducedexpression};
recall that
$u s_{\ell_k} = s_{i_k} u$ for $1 \leq k \leq M$.
It follows from Lemma \ref{vmu} (3)
that $\left(v(\mu) v(\lambda_-)^{-1} \right)m_{\lambda_-} = m_{\mu}$
and
$- \ell(v(\mu) v(\lambda_-)^{-1} ) + \ell(m_{\lambda_-})  
=  \ell(m_{\mu})$.
Furthermore, we see that
\begin{align*}
 \left( v(\mu) v(\lambda_-)^{-1} \right)   m_{\lambda_-}
&=
\left( s_{i_{K}}\cdots s_{i_1} \right)  us_{\ell_{1}} \cdots s_{\ell_L}\\
&=
u s_{\ell_{K}}\cdots s_{\ell_1} s_{\ell_1} \cdots s_{\ell_L}
=
us_{\ell_{K+1}}\cdots s_{\ell_L} \quad (\mathrm{by\ Lemma\ } \ref{lengthadditive}),
\end{align*}
and hence that 
$m_{\mu}=us_{\ell_{K+1}}\cdots s_{\ell_L}$
is a reduced expression for $m_{\mu}$.
In particular, when $\mu = \lambda$
(note that $v(\lambda)=e$), $m_{\lambda}=us_{\ell_{M+1}}\cdots s_{\ell_L}$
is a reduced expression for $m_{\lambda}$.

Also,
recall from \S 2.2 and \S 3.1 that
$\beta_k = s_{i_N} \cdots s_{i_{k+1}}\alpha_{i_k}$ for 
$1 \leq k \leq N$,
and
$\beta^\OS_k = s_{\ell_L} \cdots s_{\ell_{k+1}}\alpha^\lor_{\ell_k}$ 
for $1 \leq k \leq L$.

\begin{rem}\label{inversionsets}
Keep the notation above.
We have
\begin{align*}
\widetilde{\Delta}^+_{\aff} \cap m^{-1}_{\lambda_-} \widetilde{\Delta}^-_{\aff} &= \{ \beta^\OS_1 , \ldots , \beta^\OS_L \}, \\
\widetilde{\Delta}^+_{\aff} \cap m^{-1}_{\mu} \widetilde{\Delta}^-_{\aff} &= \{ \beta^\OS_{K+1} , \ldots , \beta^\OS_L \}, \\
\widetilde{\Delta}^+_{\aff} \cap m^{-1}_{\lambda} \widetilde{\Delta}^-_{\aff} &= \{ \beta^\OS_{M+1} , \ldots , \beta^\OS_L \}, \\
\lons \left(\Delta^+ \cap v(\mu)\Delta^-\right) &= \{ \beta_{K+1}, \ldots , \beta_{M}\}.
\end{align*}
In particular, we have 
$\widetilde{\Delta}^+_{\aff} \cap m^{-1}_{\lambda}\widetilde{\Delta}^-_{\aff}
\subset
\widetilde{\Delta}^+_{\aff} \cap m^{-1}_{\mu} \widetilde{\Delta}^-_{\aff}
\subset
\widetilde{\Delta}^+_{\aff} \cap m^{-1}_{\lambda_-} \widetilde{\Delta}^-_{\aff}$.
\end{rem}

\begin{lemm}[\normalfont{\cite[(2.4.7) (i)]{M}}]\label{shortest}
If we denote by $\phi$ the characteristic function of $\Delta^-$, i.e.,
\begin{align*}
\phi(\gamma) \eqdef
\left\{
			\begin{array}{ll}
				0  & \ \ \mathrm{if} \ \gamma \in \Delta^+  ,  \\
				1  & \ \ \mathrm{if} \ \gamma \in \Delta^-  ,
			\end{array}
\right.
\end{align*}
then we have
\begin{equation*}
\widetilde{\Delta}^+_{\aff} \cap m^{-1}_{\mu} \widetilde{\Delta}^-_{\aff}
=
\{
\alpha^\lor + a\widetilde{\delta}
\ | \
\alpha \in \Delta^-,
0 < a < \phi(v(\mu) v(\lambda_-)^{-1} \alpha) +\langle  \lambda_- ,  \alpha^\lor \rangle
\}.
\end{equation*}
\end{lemm}

\begin{rem}\label{affinereflectionorder}
Let
$\gamma_1, \gamma_2 ,\ldots, \gamma_r \in \widetilde{\Delta}^+_{\aff} \cap m^{-1}_\mu \widetilde{\Delta}^-_{\aff}$,
and 
consider the sequence 
$\left( y_0 , y_1 , \ldots , y_r ; \gamma_1 , \gamma_2, \ldots , \gamma_r \right)$
given by
$y_0=m_\mu$, and $y_i =y_{i-1}s_{\gamma_i}$ for $1 \leq i \leq r$.
Then, the sequence 
$\left( y_0 , y_1 , \ldots , y_r ; \gamma_1 , \gamma_2, \ldots , \gamma_r \right)$
is an element of 
$\B(w; m_{\mu})$
if and only if the following conditions hold:

\begin{enu}
\item
$\gamma_1\prec' \gamma_2 \prec' \cdots \prec' \gamma_r$,
where the order $\prec'$ is 
the total order on $ \widetilde{\Delta}^+_{\aff} \cap m^{-1}_\mu \widetilde{\Delta}^-_{\aff}$
introduced at the beginning of \S 3.2;

\item
$\dir(y_{i-1}) \xleftarrow{- \overline{\gamma_i}^\lor} \dir(y_{i})$
is
an edge of $\DBG$ for $1\leq i \leq r$.
\end{enu}
\end{rem}

In the following,
we will define a map
$\Xi : \B(w; m_{\mu}) \rightarrow w(\PQLS^\mu (\lambda))$ for $w \in W$.
Let $p^{\OS}_{J}$ be an arbitrary element of  $\B(w; m_{\mu})$ of the form: 
	\begin{equation*}
		p^{\OS}_{J} = \left(w m_{\mu} = z^{\OS}_0 ,  z^{\OS}_{1} , \ldots , z^{\OS}_{r}
; \beta^{\OS}_{j_1} , \beta^{\OS}_{j_2}, \ldots , \beta^{\OS}_{j_r} \right) \in \B(w; m_{\mu}),
	\end{equation*}
with $J = \{ j_1 < \cdots < j_r\} \subset \{ K+1 , \ldots , L \}$.
	We set $x_k \eqdef {\dir}(z^{\OS}_k)$ for 
	$0 \leq k \leq r$. Then, by the definition of $\B(w; m_{\mu})$, 
	\begin{equation}\label{2.15.5}
		wv(\mu)v(\lambda_-)^{-1} = \dir (wm_\mu) =  x_0 \xrightarrow{-  \overline{ \beta^{\OS}_{j_1} }^{\lor}  }   x_1
		\xrightarrow{-  \overline {\beta^{\OS}_{j_2} }^{\lor} } \cdots  \xrightarrow{ - \overline{ \beta^{\OS}_{j_r} }^{\lor} }  x_r
	\end{equation}
	is a directed path in $\DBG$; here the first equality follows from Lemma \ref{vmu}.
	We take $0 = u_0 \leq u_1 < \cdots < u_{s-1} < u_s=r$ and $0 \leq \sigma_1 < \cdots <\sigma_{s-1} < 1 = \sigma_{s}$ in such a way that
(see (\ref{C}))
	\begin{equation}\label{2.16}
		 \underbrace{0 = d_{j_1} = \cdots = d_{j_{u_1}} }_{=\sigma_0}
		< \underbrace{d_{j_{u_1 +1}} = \cdots =d_{j_{u_2}}}_{=\sigma_1} < \cdots <
		\underbrace{ d_{j_{u_{s-1}+1}} = \cdots =d_{j_r} }_{=\sigma_{s-1}} <1 = \sigma_{s};
	\end{equation}
	note that $d_{j_1}>0$ if and only if $u_1=0$.
	We set
	$w'_{p, m} \eqdef x_{u_p + m}$ for $0 \leq p \leq s-1$, $0 \leq m \leq t_p \eqdef  u_{p+1} - u_{p}$. 
	Then, by taking a subsequence of (\ref{2.15.5}), we obtain the following  directed path in $\DBG$
	for each $0 \leq p \leq s-1$:
	\begin{equation*}
		w'_{p, 0}  \xrightarrow{-  \overline{ \beta^{\OS}_{j_{u_p +1}} } ^{\lor} }   w'_{p,1} 
		\xrightarrow{-  \overline{  \beta^{\OS}_{j_{u_p +2}} }^{\lor} } \cdots  
		\xrightarrow{ - \overline{  \beta^{\OS}_{j_{u_{p+1}}} }^{\lor} }  w'_{p, t_p};
	\end{equation*}
note that $w'_{p, t_p} = w'_{p+1,0}$ for $0 \leq  p \leq s-2$.
We set $w_{p,m} = w'_{p,m}\lon$ for $0 \leq p \leq s-1, 0\leq m \leq t_p$.
	Multiplying this directed path by $\lon$ on the right, 
	we obtain the following directed path in $\DBG$
	for each $0 \leq p \leq s-1$
	(see Lemma \ref{involution} (2)):
	\begin{equation}\label{w'_p}
		w_{p,0} \eqdef
		w'_{p,0} \lon  \xleftarrow{ \lon \overline{ \beta^{\OS}_{j_{u_p +1}} }^{\lor} }  \cdots  
		\xleftarrow{ \lon \overline{ \beta^{\OS}_{j_{u_{p+1}}} }^{\lor} }   w'_{p+1}\lon  \eqdef w_{p, t_p}.
	\end{equation}
	Note that the edge labels of this directed path are increasing
in the total order $\prec$ on $\Delta^+$ introduced at the beginning of \S 2.2
(see Lemma \ref{remark2.11}), 
	  and that they lie in $\Delta^+ \setminus \Delta^+_S$;
	this property will be used to prove Proposition \ref{bijective}.
	Because
	\begin{equation*}
		\sigma_p \langle \lambda , \lon \overline{ {\beta^{\OS}_{j_u}} }   \rangle
		=
		d_{j_u} \langle \lambda , \lon \overline{ {\beta^{\OS}_{j_u}} }   \rangle
		=
		\frac{\langle \lambda_- ,  \overline{ \beta^{\OS}_{j_u}} \rangle -a_{j_u}}{\langle \lambda_- ,  \overline{ \beta^{\OS}_{j_u} }   \rangle}
		\langle \lambda_- , \overline{ \beta^{\OS}_{j_u} }   \rangle
		=
		 \langle \lambda_- ,  \overline{ \beta^{\OS}_{j_u}} \rangle -a_{j_u}  \in \mathbb{Z}
	\end{equation*}
	for $u_p +1 \leq u \leq u_{p+1}$, $0 \leq p \leq s-1$,
	we find that (\ref{w'_p}) is a  directed  path in $\DBG_{\sigma_p \lambda}$ for $0 \leq p \leq s-1$.
	Thus we obtain 
	\begin{equation}\label{D}
\widetilde{\eta}\eqdef
 \left(
\left( 
w_{p, 0} \xleftarrow{\beta_{p,1}} \cdots \xleftarrow{\beta_{p,t_p}}w_{p, t_p}
\right)_{p=0, \ldots , s-1};
\sigma_0, \ldots , \sigma_s \right)
 \in w(\PQLS^\mu(\lambda));
	\end{equation}
	here we set $\beta_{p,m}\eqdef \lon \overline{\beta^\OS_{u_p+m}}^\lor$ for $p \leq p \leq s-1$, $1 \leq m \leq t_p$.
Indeed,
condition (i) in Definition \ref{defi:pQLS} has been checked above;
condition (ii) is obvious since $w'_{p, t_p} = w'_{p+1,0}$ for $0 \leq p \leq s-2$. 
Also,
for each $1 \leq m \leq  t_p$, we see that $\beta_{0,m}= \lon \overline{\beta^\OS_j}^\lor$ for some $1 \leq j \leq M$ 
and $\lon\overline{\beta^\OS_j}^\lor = \beta_{j}$ by Remark 
\ref{rem:2.15} (2) and (\ref{3.1.9});
here,
$\beta_j \in \lons (\Delta^+ \cap v(\mu)^{-1} \Delta^-)$ by Remark \ref{inversionsets}.
Hence we obtain 
$\lons \beta_{0,m}\in \Delta^+ \cap v(\mu)^{-1} \Delta^-$, which implies that $\widetilde{\eta}$ satisfies condition (iii).
Also,
since
$x_0 = wv(\mu)v(\lambda_-)^{-1}$, we have 
\begin{equation*}
w_{0,0} = x_0 \lon = wv(\mu)v(\lambda_-)^{-1}\lon = w v(\mu) \lons 
\end{equation*}
by Lemma \ \ref{vmu} \ (2).
Therefore, the pair
	\begin{equation*}
w^{-1} \widetilde{\eta}=
 \left(
\left( 
w^{-1} w_{k, 0} \xleftarrow{\beta_{k,1}} \cdots \xleftarrow{\beta_{k,t_k}}w^{-1} w_{k, t_k}
\right)_{k=0, \ldots , s-1};
\sigma_0, \ldots , \sigma_s \right)
	\end{equation*}
satisfies conditions (i)-(iv), and hence $w^{-1}\widetilde{\eta}\in \PQLS^\mu(\lambda)$.

We now define $\Xi (p^{\OS}_{J}) \eqdef \widetilde{\eta}$.

\begin{pro}\label{bijective}
		The map $\Xi:\B(w; m_\mu) \rightarrow w(\PQLS^\mu(\lambda) )$
is bijective.
\end{pro}

\begin{proof}
It is clear that the map  $\Xi:\B(w; m_\mu) \rightarrow w(\PQLS^\mu(\lambda) )$ is injective.
We will prove that the map  $\Xi:\B(w; m_\mu) \rightarrow w(\PQLS^\mu(\lambda) )$ is surjective.
Let us take an arbitrary element
	\begin{equation}
\widetilde{\eta} = 
 \left(
\left( 
v_{p, 0} \xleftarrow{\beta_{p,1}} \cdots \xleftarrow{\beta_{p,t_p}}v_{p, t_p}
\right)_{p=0, \ldots , s-1};
\tau_0, \ldots , \tau_s \right)
 \in w(\PQLS(\lambda)).
	\end{equation}
For $0 \leq p \leq s-1$, $1 \leq m \leq t_p$,
we set
$a_{p,m}\eqdef  (\tau_p -1)\langle \lambda_-, -\lon \beta_{p,m}^\lor \rangle $ 
and
$\widetilde{\beta}_{p,m} \eqdef
\lon \beta_{p,m}^\lor+
a_{p,m} \widetilde{\delta}$.

\begin{nclaim}
$\widetilde{\beta}_{p,m} \in \widetilde{\Delta}_\aff^+ \cap m_{\mu}^{-1}\widetilde{\Delta}_\aff^-$.
\end{nclaim}

\noindent
$Proof \ of \ Claim \ 1.$
If $p \neq 0$,
then 
$a_{p,m}=(\tau_p -1)\langle \lambda_-, -\lon \beta_{p,m}^\lor \rangle 
<\langle \lambda_-, \lon \beta_{p,m}^\lor \rangle $
since $\tau_p \neq 0$,
and hence 
$a_{p,m}<\phi(v(\mu)v(\lambda_-)^{-1}\lon \beta_{p,m} ) + \langle \lambda_-, \lon \beta_{p,m}^\lor \rangle $.
If $p=0$, 
then $\lons \beta_{p,m} \in \Delta^+ \cap v(\mu)^{-1} \Delta^-$
by the definition of $\PQLS^{\mu}(\lambda)$.
Thus,
$v(\mu)v(\lambda_-)^{-1} (\lon \beta_{p,m} ) = v(\mu)\lons \beta_{p,m} \in \Delta^-$,
and hence
$\phi(v(\mu)v(\lambda_-)^{-1} (\lon \beta_{p,m} ))=1$.
Therefore, in both cases, we deduce that 
$a_{p,m}=
(\tau_p -1)
\langle \lambda_-, -\lon \beta_{p,m}^\lor \rangle <\phi(v(\mu)v(\lambda_-)^{-1} (\lon \beta_{p,m} ))+\langle \lambda_-, \lon \beta_{p,m}^\lor \rangle $,
and hence that 
$\widetilde{\beta}_{p,m} \in \widetilde{\Delta}_\aff^+ \cap m_{\mu}^{-1}\widetilde{\Delta}_\aff^-$
by Lemma \ref{shortest}.
This proves the claim.
\bqed

\begin{nclaim}
We have
\begin{equation*}
\widetilde{\beta}_{0,1} \prec'  \cdots \prec' \widetilde{\beta}_{0, t_0} 
\prec'  
\widetilde{\beta}_{1,1} \prec' \cdots \prec' \widetilde{\beta}_{s-1,t_{s-1}},
\end{equation*}
where $\prec'$ denotes the total order on $\widetilde{\Delta}^+_{\aff} \cap m^{-1}_{\lambda_-} \widetilde{\Delta}^-_{\aff}$
introduced at the beginning of \S 3.2;
we choose 
$J' = \{ j_1 , \ldots , j'_{r'} \} \subset \{ K+1 , \ldots , L \}$ in such way that
\begin{align*}
\left( \beta^{\OS}_{j'_1} , \cdots , \beta^{\OS}_{j'_{r'}} \right)
=
\left(  
\widetilde{\beta}_{0,1} ,  \cdots , \widetilde{\beta}_{0,t_0} ,
\widetilde{\beta}_{1,1} , \cdots , \widetilde{\beta}_{s-1,t_{s-1}}
\right).
\end{align*}

\end{nclaim}

\noindent $Proof \ of \ Claim \ 2.$
It suffices to show the following:

(i)
for $0 \leq p \leq s-1$ and $1 \leq m < t_p$,
we have
$  \widetilde{\beta}_{p,m}\prec' \widetilde{\beta}_{p,m+1}$;

(ii)
for $0 \leq p \leq s-2$,
we have
$\widetilde{\beta}_{p,t_p} \prec' \widetilde{\beta}_{p+1,1}$.

(i)
Since
$\frac{\langle \lambda_- , \lon \beta_{p,m}^{\lor}\rangle - a_{p,m}}{\langle \lambda_- , \lon \beta_{p,m}^{\lor}\rangle}=\tau_p$ and
$\frac{\langle \lambda_- , \lon \beta_{p,m+1}^{\lor}\rangle - a_{p,m+1}}{\langle \lambda_- , \lon \beta_{p,m+1}^{\lor}\rangle}=\tau_p$,
it follows that 
\begin{align*}
\Phi (\widetilde{\beta}_{p,m}) &=
 (\tau_p, \beta_{p,m} ), \\
\Phi (\widetilde{\beta}_{p,m+1}) &=
 (\tau_p, \beta_{p,m+1} ).
\end{align*}
Therefore, the first component of $ \Phi (\widetilde{\beta}_{p,m})$ is equal to that of $\Phi (\widetilde{\beta}_{p,m+1})  $.
Also, since $ \beta_{p,m}  \succ \beta_{p,m+1} $,
we have $ \Phi (\widetilde{\beta}_{p,m}) < \Phi (\widetilde{\beta}_{p,m+1})  $.
This implies that $  \widetilde{\beta}_{p,m}\prec' \widetilde{\beta}_{p,m+1}$.

(ii)
The proof of (ii) is similar to that of (i).
The first components of $\Phi (\widetilde{\beta}_{p,t_p}) $ and $\Phi (\widetilde{\beta}_{p+1,1}) $ are $\tau_p$ and $\tau_{p+1}$, respectively.
Since $\tau_p < \tau_{p+1}$,
we have $\Phi (\widetilde{\beta}_{p,t_p}) < \Phi (\widetilde{\beta}_{p+1,1}) $.
This implies that
$\widetilde{\beta}_{p,t_p} \prec' \widetilde{\beta}_{p+1,1}$.
This proves the claim.
\bqed

Since $J' = \{ j_1 , \ldots , j'_{r'} \} \subset \{ K+1 , \ldots , L \}$,
we can define 
\begin{equation*}
p^{\OS}_{J'} 
= \left( wm_{\mu} = z^{\OS}_0 ,  z^{\OS}_{1} , \ldots , z^{\OS}_{r'}
; \beta^{\OS}_{j'_1} , \beta^{\OS}_{j'_2}, \ldots , \beta^{\OS}_{j'_{r'}} \right) 
\end{equation*}
by: 
$z^{\OS}_{0}=wm_{\mu}$, $z^{\OS}_{k}=z^\OS_{k-1} s_{\beta^{\OS}_{j'_{k}}}$ for $1\leq k \leq r'$;
it follows from Remark \ref{affinereflectionorder} and Claim 2 that
 $p^{\OS}_{J'} \in \B(w; m_{\mu})$.

\begin{nclaim}
$\Xi (p^\OS_{J'}) = \widetilde{\eta}$.
\end{nclaim}

\noindent $Proof \ of \ Claim \ 3.$
In the following description of $ p^\OS_{J'}$,
we employ the notation  $u_p$, $\sigma_p$, and $w_{p,m}$
used  in the definition of $\Xi (p^\OS_{J})$.

For $1 \leq k \leq r'$, 
we see that 
\begin{equation*}
d_{j'_k}=
1 + \frac{\dg(\beta^{\OS}_{j'_k})}{\langle \lambda_- , -\overline{ \beta^{\OS}_{j'_k} } \rangle}
=
1 + 
\frac{\dg(\widetilde{\beta}_{p,m})}{\langle \lambda_- , -\overline{\widetilde{\beta}_{p,m}} \rangle}
=
1 + \frac{a_{p,m}}{\langle \lambda_- , \beta_{p,m}^\lor \rangle}
=
d_{p,m}
.
\end{equation*}
Therefore,
the sequence (\ref{2.16}) determined by $p^{\OS}_{J'}$ is 
\begin{equation}\label{2.17}
		\underbrace{ 0 = d_{0,1} = \cdots d_{0,t_0} }_{=\tau_0}
		< \underbrace{d_{1,1} = \cdots =d_{1, t_1}}_{=\tau_1}
		 < \cdots <
		\underbrace{ d_{s-1,1} = \cdots =d_{s-1,t_{s-1}} }_{=\tau_{s-1}} <1 = \tau_{s} = \sigma_s.
	\end{equation}
Because the sequence (\ref{2.17}) of rational numbers is identical to the sequence (\ref{2.16}) for $p_{J'}^\OS$,
we deduce that 
$u_{p+1} - u_p = t_p$ for $0 \leq p \leq s-1$,
$\beta^\OS_{j'_{u_p + m}} = \widetilde{\beta}_{p,m}$ for $0 \leq p \leq s-1$, $1 \leq m \leq u_{p+1} - u_p$,
and
$\sigma_p = \tau_{p}$ for $0 \leq p \leq s$.
Hence it follows that 
$\lon \overline{ \beta^{\OS}_{j'_{u_p +m}} }^\lor = \beta_{p,m}$ and
 $w_{p, m} = v_{p,m}$.
Thus we conclude that $\Xi (p^\OS_{J'}) = \widetilde{\eta}$.
This proves the claim.
\bqed

Thus, we have proved that the map $\Xi:\B(w; m_\mu) \rightarrow w(\PQLS^\mu(\lambda))$ is surjective.
This completes the proof of Proposition \ref{bijective}.
\end{proof}

In order to prove Theorem \ref{theorem_graded_character}, 
we need the following lemmas.

\begin{lemm}
Let $p^{\OS}_{J} \in \B(w; m_{\mu})$, and set
\begin{equation*}
\Xi (p^{\OS}_{J})=
 \left(
\left( 
w_{k, 0} \xleftarrow{\beta_{k,1}} \cdots \xleftarrow{\beta_{k,t_k}}w_{k, t_k}
\right)_{k=0, \ldots , s-1};
\sigma_0, \ldots , \sigma_s \right)
.
\end{equation*}
Then, we have 
		$\ell(\dir (\ed(p^{\OS}_{J}))) - \ell(\dir(wm_{\mu})) = \ell(w v(\mu) \lons) - \ell(w_{s-1, t_{s-1}})$.
\end{lemm}

\begin{proof}
Recall from Lemma \ref{vmu} (1), (2) that $\dir(m_\mu) = v(\mu)v(\lambda_-)^{-1}= v(\mu)\lons \lon$.
Also, by the definition of $\Xi$, we have $\dir(\ed(p^{\OS}_{J})))=w_{s-1, t_{s-1}}\lon$.
Therefore,
we see that 
\begin{align*}
\ell(\dir (\ed(p^{\OS}_{J}))) - \ell(\dir(wm_{\mu})) &=\ell (w_{s-1, t_{s-1}}\lon)  -
\ell(wv(\mu)\lons \lon)  \\ &=
 \ell(w v(\mu) \lons) - \ell(w_{s-1, t_{s-1}}).
\end{align*}
This proves the lemma.
\end{proof}

\begin{lemm}\label{qls_qb}
Let $p^{\OS}_{J} \in \B(w; m_{\mu})$.
Then, the following hold: 

\begin{enu}
\item
		$\wt (\ed (p^{\OS}_{J})) = \wt (\pr (\Xi(p^\OS_{J}) ))${\rm;}

\item
$
 \mathcal{R}(\Xi(p^{\OS}_{J}))) =  \mathcal{L}(p^{\OS}_{J}).
$   \end{enu}
\end{lemm}

\setcounter{nclaim}{0} 

\begin{proof}
We proceed by induction on $ \#J$.
If $J=\emptyset$, then it is obvious that $\mathcal{L} (p^{\OS}_{J})=\mathcal{R}(\Xi(p^{\OS}_{J}))=1$ and 
${\wt}(\ed (p^{\OS}_{J})) = \wt(\pr (\Xi(p^{\OS}_{J}) ))=w\mu$,
since
$\Xi(p^{\OS}_{J})=(wv(\mu)\lons; 0,1)$.
	
Let	$J=\{ j_1 <j_2 <\cdots <j_r\}$ with $r \geq 1$, and set
	$K \eqdef J\setminus \{ j_r \}$. 
	Assume that $\Xi (p^{\OS}_{K})$ is of the form: 
\begin{equation*}
\Xi(p^{\OS}_K)=
\widetilde{\eta}=
 \left(
\left( 
w_{p, 0} \xleftarrow{\beta_{p,1}} \cdots \xleftarrow{\beta_{p,t_p}}w_{p, t_p}
\right)_{p=0, \ldots , s-1};
\sigma_0, \ldots , \sigma_s \right)
;
\end{equation*}
note that
$\dir(\ed(p^\OS_K))=w_{s-1, t_{s-1}} \lon$
 and 
$w_{0,0}=v(\mu) \lons$
by the definition of $\Xi$.

If $d_{j_{r}}=d_{j_{r-1}}={\sigma}_{s-1}$, 
then $\{ d_{j_1} \leq \cdots \leq d_{j_{r-1}}\leq d_{j_r} \}
=
\{ d_{j_1} \leq \cdots \leq d_{j_{r-1}} \} $,
and hence
\begin{equation*}
\Xi(p^{\OS}_J)=
 \left(
\left( 
w'_{p, 0} \xleftarrow{\beta_{p,1}} \cdots \xleftarrow{\beta_{p,t'_p}}w'_{p, t'_p}
\right)_{p=0, \ldots , s-1}
;
\sigma_0, \ldots , \sigma_s \right)
,
\end{equation*}
where 
\begin{align*}
t'_p &=t_p &\text{ for }& 0 \leq p \leq s-2,\\
t'_{s-1} &= t_{s-1}+1, \\
w'_{p, m}&= w_{p,m} &\text{ for }& 0 \leq p \leq s-1, 1 \leq m \leq t_{p-1},\\
w'_{s-1,t'_{s-1}} &= w_{s-1,t_{s-1}}s_{\lon \overline {\beta^{\OS}_{j_r} }  },\\
\beta_{s-1,t'_{s-1}}&= \lon \overline {\beta^{\OS}_{j_r} }^\lor.
\end{align*}

If $d_{j_{r}}>d_{j_{r-1}}={\sigma}_{s-1}$,
then
$\{ d_{j_1} \leq \cdots \leq d_{j_{r-1}}\leq d_{j_r} \}
= 
\{ d_{j_1} \leq \cdots \leq d_{j_{r-1}}< d_{j_r} \}
$, and hence
\begin{equation*}
\Xi(p^{\OS}_J)=
 \left(
\left( 
w'_{p, 0} \xleftarrow{\beta_{p,1}} \cdots \xleftarrow{\beta_{p,t'_p}}w'_{p, t'_p}
\right)_{p=0, \ldots , s}
;
\sigma_0, \ldots , \sigma_{s-1}, d_{j_r}, \sigma_s \right)
,
\end{equation*}
where 
\begin{align*}
t'_p &=t_p &\text{ for }& 0 \leq p \leq s-1,\\
t'_{s} &= 1, \\
w'_{p, m}&= w_{p,m} &\text{ for }& 0 \leq p \leq s-1, 1 \leq m \leq t_{p-1},\\
w'_{s, 0}&= w_{s-1,t_{s-1}},\\
w'_{s,t'_{s}} &= w_{s-1,t_{s-1}}s_{\lon \overline {\beta^{\OS}_{j_r} }  },\\
\beta_{s,t'_{s}}&
=\beta_{s,1}= \lon (\overline {\beta^{\OS}_{j_r} } )^\lor.
\end{align*}

For the induction step, it suffices to show the following claims.

\begin{nclaim}
(1) We have
	\begin{equation*}
		\wt ( \pr(\Xi (p^{\OS}_{J}) )) 
		= \wt (\pr( \Xi (p^{\OS}_{K}) ))-
		a_{j_{r}}w_{s-1, t_{s-1}}\lon \overline {\beta^{\OS}_{j_r } }^\lor.
	\end{equation*}

(2) We have
\begin{equation*}
		\mathcal{R}(\Xi (p^{\OS}_{J}) )=
\left\{
			\begin{array}{l}
				\mathcal{R}(\Xi (p^{\OS}_{K}) )\frac{
				t^{-\half}(1-t)
				}{1- q^{\dg(\beta_{j_r}^\OS)}t^{\langle \rho, - \overline{ \beta^\OS_{j_r} }\rangle}}    \\ \hspace{1.5cm} \mbox{if} \ 
\dir(\ed(p^{\OS}_{K})) \xrightarrow{ - \overline{ \beta^\OS_{j_r} }^\lor} \dir(\ed(p^{\OS}_{J}))\mbox{ is a Bruhat edge,} \\[3mm]
				\mathcal{R}(\Xi (p^{\OS}_{K}) )\frac{q^{\dg(\beta_{j_r}^\OS)}t^{\langle \rho, - \overline{ \beta^\OS_{j_r} }\rangle-\half}(1-t)}{1- q^{\dg(\beta_{j_r}^\OS)}t^{\langle \rho, - \overline{ \beta^\OS_{j_r} }\rangle}}   \\ \hspace{1.5cm} \mbox{if} \ 
\dir(\ed(p^{\OS}_{K})) \xrightarrow{ - \overline{ \beta^\OS_{j_r} }^\lor} \dir(\ed(p^{\OS}_{J}))\mbox{ is a quantum edge;} \\
			\end{array}
\right.
	\end{equation*}
note that $\dir(\ed(p^{\OS}_{J}))=\dir(\ed(p^{\OS}_{K}))s_{\overline{ \beta^\OS_{j_r} }}$.
\end{nclaim}

\begin{nclaim}

(1) We have
	\begin{equation*}
		\wt ( \ed (p^{\OS}_{J}) ) 
		= \wt ( \ed (p^{\OS}_{K}) )-
		a_{j_{r}}w_{s-1, t_{s-1}}\lon \overline{\beta^{\OS}_{j_r } }^\lor.
	\end{equation*}

(2) We have
\begin{equation*}
		\mathcal{L}(\Xi (p^{\OS}_{J}) )=
\left\{
			\begin{array}{l}
				\mathcal{L}(\Xi (p^{\OS}_{K}) )\frac{t^{-\half}(1-t)}{1- q^{\dg(\beta_{j_r}^\OS)}t^{\langle \rho, - \overline{ \beta^\OS_{j_r} }\rangle}}    \\ \hspace{1.5cm} \mbox{if} \ 
\dir(\ed(p^{\OS}_{K})) \xrightarrow{ - \overline{ \beta^\OS_{j_r} }^\lor} \dir(\ed(p^{\OS}_{J}))\mbox{ is a Bruhat edge,} \\[3mm]
				\mathcal{L}(\Xi (p^{\OS}_{K}) )\frac{q^{\dg(\beta_{j_r}^\OS)}t^{\langle \rho, - \overline{ \beta^\OS_{j_r} }\rangle-\half}(1-t)}{1- q^{\dg(\beta_{j_r}^\OS)}t^{\langle \rho, - \overline{ \beta^\OS_{j_r} }\rangle}}   \\ \hspace{1.5cm} \mbox{if} \ 
\dir(\ed(p^{\OS}_{K})) \xrightarrow{ - \overline{ \beta^\OS_{j_r} }^\lor} \dir(\ed(p^{\OS}_{J}))\mbox{ is a quantum edge.} \\
			\end{array}
\right.
	\end{equation*}
\end{nclaim}

\noindent $Proof \ of \ Claim \ 1.$
(1)
If $d_{j_{r}}=d_{j_{r-1}}={\sigma}_{s-1}$, then we compute:
\begin{align*}
		&\wt ( \pr(\Xi (p^{\OS}_{J})) ) \\
		&=\sum_{p=1}^{s} ({\sigma}_{p} - {\sigma}_{p-1} ) w'_{p-1, t'_{p-1}} \lambda \\
		&=\sum_{p=1}^{s-1} ({\sigma}_{p} - {\sigma}_{p-1} ) w_{p-1, t_{p-1}}\lambda + (\sigma_s -\sigma_{s-1})
			 w_{s-1, t_{s-1}}s_{\lon \overline {\beta^{\OS}_{j_r } }  } \lambda
			 \\
		&= \sum_{p=1}^{s} ({\sigma}_{p} - {\sigma}_{p-1} ) w_{p-1, t_{p-1}} \lambda 
		+ (\sigma_s- \sigma_{s-1}) w_{s-1, t_{s-1}}s_{\lon \overline {\beta^{\OS}_{j_r } }  }\lambda - (\sigma_s- \sigma_{s-1}) w_{s-1, t_{s-1}}\lambda \\
		&=
		 \sum_{p=1}^{s} ({\sigma}_{p} - {\sigma}_{p-1} ) w_{p-1, t_{p-1}}\lambda 
		+ (1-d_{j_{r}}) w_{s-1, t_{s-1}}s_{\lon \overline {\beta^{\OS}_{j_r } }  }\lambda - (1-d_{j_{r}}) w_{s-1, t_{s-1}}\lambda \\
		&\hphantom{=\sum_{p=1}^{s} ({\sigma}_{p} - {\sigma}_{p-1} ) w_{p-1, t_{p-1}}\lambda + (1-d_{j_{r}}) w_{s-1, t_{s-1}}s_{\lon \overline {\beta^{\OS}_{j_r } }  }\lambda }(\mathrm{since}\ \sigma_s=1,\ d_{j_r}=\sigma_{s-1})
	\end{align*}
If $d_{j_{r}}>d_{j_{r-1}}={\sigma}_{s-1}$, then we compute:
\begin{align*}
		&\wt ( \pr(\Xi (p^{\OS}_{J}) )) \\
		&=\sum_{p=1}^{s-1} ({\sigma}_{p} - {\sigma}_{p-1} )  w_{p-1, t_{p-1}} \lambda 
	+(d_{j_r}-\sigma_{s-1})w_{s-1, t_{s-1}} \lambda
	+(\sigma_s - d_{j_r})
			 w_{s-1, t_{s-1}}s_{\lon \overline {\beta^{\OS}_{j_r } }  } \lambda
\\
		&= \sum_{p=1}^{s} ({\sigma}_{p} - {\sigma}_{p-1} ) w_{p-1,t_{p-1}}\lambda + (d_{j_r}-\sigma_{s-1})w_{s-1, t_{s-1}} \lambda 
				\\	
		&\hphantom{ \sum_{p=1}^{s} ({\sigma}_{p} - {\sigma}_{p-1} ) w_{p-1,t_{p-1}}\lambda +}
			+
			(\sigma_s - d_{j_r})
			w_{s-1, t_{s-1}}s_{\lon \overline {\beta^{\OS}_{j_r } }  }\lambda 
			 -({\sigma}_{s} - {\sigma}_{s-1} ) w_{s-1, t_{s-1}}\lambda
\\
&= \sum_{p=1}^{s} ({\sigma}_{p} - {\sigma}_{p-1} ) w_{p-1, t_{p-1}} \lambda 
		+ (\sigma_s- d_{j_r}) w_{s-1, t_{s-1}} s_{\lon \overline {\beta^{\OS}_{j_r } }  }\lambda - (\sigma_s-d_{j_r}) w_{s-1, t_{s-1}}\lambda \\
		&\overset{\sigma_s=1}{=}
		 \sum_{p=1}^{s} ({\sigma}_{p} - {\sigma}_{p-1} ) w_{p-1, t_{p-1}}\lambda 
		+ (1-d_{j_{r}}) w_{s-1, t_{s-1}}s_{\lon \overline {\beta^{\OS}_{j_r } }  }\lambda - (1-d_{j_{r}}) w_{s-1, t_{s-1}}\lambda .
	\end{align*}
In both cases above, since
$\wt(\pr (\Xi(p^{\OS}_K))) =\sum_{p=1}^{s} ({\sigma}_{p} - {\sigma}_{p-1} )  w_{p-1, t_{p-1}} \lambda  $,
and since
\begin{align*}
& (1-d_{j_{r}}) w_{s-1, t_{s-1}}s_{\lon \overline {\beta^{\OS}_{j_r } }  }\lambda - (1-d_{j_{r}}) w_{s-1, t_{s-1}}\lambda \\
 &=-(1-d_{j_{r}}) w_{s-1, t_{s-1}} {\langle \lambda,  \lon \overline {\beta^{\OS}_{j_r } }   \rangle} \lon \overline { \beta^{\OS}_{j_r } }^\lor  \\
&=
- \frac{a_{j_r}}{\langle \lambda_-, \overline{ \beta^{\OS}_{j_r} }     \rangle}
{\langle \lambda_-, \overline{ \beta^{\OS}_{j_r} }    \rangle}
 w_{s-1, t_{s-1}} \lon \overline { \beta^{\OS}_{j_r } }^\lor \quad
({\rm by \ Remark} \ \ref{2.15}) 
 \\
&=-
 a_{j_{r}}w_{s-1, t_{s-1}}\lon \overline { \beta^{\OS}_{j_r } }^\lor,
\end{align*}
it follows that
	\begin{align*}
		&\wt ( \pr(\Xi (p^\OS_{J}))) \\
		&= \sum_{p=1}^{s} ({\sigma}_{p} - {\sigma}_{p-1} ) w_{p-1, t_{p-1}}\lambda 
		+ (1-d_{j_{r}}) w_{s-1, t_{s-1}}s_{\lon \overline {\beta^{\OS}_{j_r } }  }\lambda - (1-d_{j_{r}}) w_{s-1, t_{s-1}}\lambda \\
		&= \wt ( \pr(\Xi (p^\OS_{K})))- a_{j_{r}}w_{s-1, t_{s-1}}\lon\overline {\beta^{\OS}_{j_r } }^\lor.
	\end{align*}

(2) We have 

\begin{align*}
\mathcal{R}(\Xi (p^{\OS}_{J}) ) 
&=\left(\prod_{p=0}^{s-1} \prod_{m=1}^{t_p}
\wt_{\sigma_p \lambda}(w_{p, m-1} \xleftarrow{\beta_{p,m}} w_{p, m}) \right)
\\
&\hspace{3cm} \times \deg_{d_{j_r}\lambda}\left(w_{s-1, t_{s-1}}  \xleftarrow{ \lon\overline{ \beta^\OS_{j_r} }^\lor} w_{s-1, t_{s-1}}s_{\lon \overline {\beta^{\OS}_{j_r } }  }\right)
\\
&=
\mathcal{R}(\Xi (p^{\OS}_{K}) )
\dg_{d_{j_r}\lambda}\left(w_{s-1, t_{s-1}}  \xleftarrow{ \lon \overline{ \beta^\OS_{j_r} }^\lor} w_{s-1, t_{s-1}}s_{\lon \overline {\beta^{\OS}_{j_r } }  }\right).
\end{align*}
If $\dir(\ed(p^{\OS}_{K})) \xrightarrow{ - \overline{ \beta^\OS_{j_r} }^\lor} \dir(\ed(p^{\OS}_{J}))$ is a Bruhat (resp., quantum) edge,
then $w_{s-1, t_{s-1}}  \xleftarrow{ \lon \overline{ \beta^\OS_{j_r} }^\lor} w_{s-1, t_{s-1}}s_{\lon \overline {\beta^{\OS}_{j_r } }  }$
is also a Bruhat (resp., quantum) edge by Lemma \ref{involution} (2),
since $\dir(\ed(p^{\OS}_{K}))=w_{s-1, t_{s-1}}\lon$ and 
\begin{equation*}
\dir(\ed(p^{\OS}_{J}))=\dir(\ed(p^{\OS}_{J})s_{ {\beta^{\OS}_{j_r } }  })=w_{s-1, t_{s-1}}\lon s_{\overline {\beta^{\OS}_{j_r } }  }.
\end{equation*}
Hence we see that 
\begin{equation*}
		\mathcal{R}(\Xi (p^{\OS}_{J}) )=
\left\{
			\begin{array}{l}
				\mathcal{R}(\Xi (p^{\OS}_{K}) )\frac{t^{-\half}(1-t)}{1- q^{(1-d_{j_r})\langle \lambda, \lon \overline{ \beta^\OS_{j_r} }\rangle}t^{\langle \rho, \lon \overline{ \beta^\OS_{j_r} }\rangle}}   \\ \hspace{1.5cm}     \mbox{if} \ 
\dir(\ed(p^{\OS}_{K})) \xrightarrow{ - \overline{ \beta^\OS_{j_r} }^\lor} \dir(\ed(p^{\OS}_{J}))\mbox{ is a Bruhat edge,} \\[3mm]
				\mathcal{R}(\Xi (p^{\OS}_{K}) )\frac{q^{(1-d_{j_r})\langle \lambda, \lon \overline{ \beta^\OS_{j_r} }\rangle}t^{\langle \rho, \lon \overline{ \beta^\OS_{j_r} }\rangle-\half}(1-t)}{1- q^{(1-d_{j_r})\langle \lambda, \lon \overline{ \beta^\OS_{j_r} }\rangle}t^{\langle \rho, \lon \overline{ \beta^\OS_{j_r} }\rangle}}   \\ \hspace{1.5cm}     \mbox{if} \ 
\dir(\ed(p^{\OS}_{K})) \xrightarrow{ - \overline{ \beta^\OS_{j_r} }^\lor} \dir(\ed(p^{\OS}_{J}))\mbox{ is a quantum edge.} \\
			\end{array}
\right.
	\end{equation*}
Here, 
${\langle \rho, - \overline{ \beta^\OS_{j_r} }\rangle}={\langle \rho, \lon \overline{ \beta^\OS_{j_r} }\rangle}$
since $\lon \rho = -\rho$.
Also,
since
$d_{j_r}=
\frac{\langle \lambda_- ,  \overline{\beta^{\OS}_{j_r}}\rangle - a_{j_r}}{\langle \lambda_- ,  \overline{\beta^{\OS}_{j_r} }\rangle}  $,
we have
$(1-d_{j_r})\langle \lambda, \lon \overline{ \beta^\OS_{j_r} }\rangle
=
\left(1-\frac{\langle \lambda , \lon \overline{\beta^{\OS}_{j_r}}\rangle - a_{j_r}}{\langle \lambda , \lon  \overline{\beta^{\OS}_{j_r} }\rangle} \right)
\langle \lambda, \lon \overline{ \beta^\OS_{j_r} }\rangle = a_{j_r}=\dg(\beta^{\OS}_{j_r})$.
Therefore, we deduce that 
\begin{equation*}
		\mathcal{R}(\Xi (p^{\OS}_{J}) )=
\left\{
			\begin{array}{l}
				\mathcal{R}(\Xi (p^{\OS}_{K}) )\frac{t^{-\half}(1-t)}{1- q^{\dg(\beta_{j_r}^\OS)}t^{\langle \rho, - \overline{ \beta^\OS_{j_r} }\rangle}}    \\ \hspace{1.5cm}  \mbox{if} \ 
\dir(\ed(p^{\OS}_{K})) \xrightarrow{ - \overline{ \beta^\OS_{j_r} }^\lor} \dir(\ed(p^{\OS}_{J}))\mbox{ is a Bruhat edge,} \\[3mm]
				\mathcal{R}(\Xi (p^{\OS}_{K}) )\frac{q^{\dg(\beta_{j_r}^\OS)}t^{\langle \rho, - \overline{ \beta^\OS_{j_r} }\rangle-\half}(1-t)}{1- q^{\dg(\beta_{j_r}^\OS)}t^{\langle \rho, - \overline{ \beta^\OS_{j_r} }\rangle}}    \\ \hspace{1.5cm}    \mbox{if} \ 
\dir(\ed(p^{\OS}_{K})) \xrightarrow{ - \overline{ \beta^\OS_{j_r} }^\lor} \dir(\ed(p^{\OS}_{J}))\mbox{ is a quantum edge.} \\
			\end{array}
\right.
	\end{equation*}
This proves Claim 1 (2).
\bqed

\noindent $Proof \ of \ Claim \ 2.$
Let us prove part (1).
Note that $\ed (p^{\OS}_{J}) ) =  \ed (p^{\OS}_{K})   s_{\beta^{\OS}_{j_r }}$,
and that
\begin{equation*}
 \ed (p^{\OS}_{K}) = t(\wt ( \ed (p^{\OS}_{K}) ) ) \dir  ( \ed (p^{\OS}_{K}) )  = t(\wt ( \ed (p^{\OS}_{K}) ) )w_{s-1, t_{s-1}} \lon ;
\end{equation*}
the second equality follows from the comment at the beginning of the proof of Lemma \ref{qls_qb}.
Also, we have
$s_{\beta^{\OS}_{j_r }} 
= s_{a_{j_r}\tilde{\delta}+\overline {\beta^{\OS}_{j_r } }}
=t\left(- a_{j_r}\overline {\beta^{\OS}_{j_r }}^\lor \right) s_{\overline {\beta^{\OS}_{j_r }}}$.
Combining these, we see that 
\begin{align*}
 \ed (p^{\OS}_{J}) 
&=
\left(
 t(\wt ( \ed (p^{\OS}_{K}) ) )w_{s-1, t_{s-1}} \lon
\right)
\left(
 t\left(-a_{j_r}\overline {\beta^{\OS}_{j_r }}^\lor \right)s_{\overline {\beta^{\OS}_{j_r }}}
\right) \\
&=t\left( \wt ( \ed (p^{\OS}_{K}) ) -a_{j_r}w_{s-1, t_{s-1}} \lon \overline {\beta^{\OS}_{j_r }}^\lor \right)
w_{s-1, t_{s-1}} \lon s_{\overline {\beta^{\OS}_{j_r }}},
\end{align*}
and hence that 
\begin{equation*}
\wt( \ed (p^{\OS}_{J}) )
=
\wt ( \ed (p^{\OS}_{K}) ) -a_{j_r}w_{s-1, t_{s-1}} \lon \overline {\beta^{\OS}_{j_r }}^\lor .
\end{equation*}

(2)
If $\dir(\ed(p^{\OS}_{K})) \xrightarrow{ - \overline{ \beta^\OS_{j_r} }^\lor} \dir(\ed(p^{\OS}_{J}))$ is a Bruhat edge,
then $J^{-}=K^{-}$ and $J^+=K^+ \sqcup \{j_r\}$.
Therefore, we see that
\begin{align*}
\mathcal{L}(p^{\OS}_{J})&=
 t^{-\half\#J}(1-t)^{\# J}
 \prod_{j \in J^+}\frac{1}{1- q^{\dg(\beta_j^\OS)}t^{\langle \rho, - \overline{ \beta^\OS_j }\rangle}}
\prod_{j \in J^-}\frac{q^{\dg(\beta_j^\OS)}t^{\langle \rho, - \overline{ \beta^\OS_j }\rangle}}{1- q^{\dg(\beta_j^\OS)}t^{\langle \rho, - \overline{ \beta^\OS_j }\rangle}}\\
&=
 t^{-\half\# K}(1-t)^{\# K}
 \prod_{j \in K^+}\frac{1}{1- q^{\dg(\beta_j^\OS)}t^{\langle \rho, - \overline{ \beta^\OS_j }\rangle}}
\\
&\hspace{2cm}\times 
\left( \prod_{j \in K^-}\frac{q^{\dg(\beta_j^\OS)}t^{\langle \rho, - \overline{ \beta^\OS_j }\rangle}}{1- q^{\dg(\beta_j^\OS)}t^{\langle \rho, - \overline{ \beta^\OS_j }\rangle}} \right)
\frac{t^{-\half}(1-t)}{1- q^{\dg(\beta_{j_r}^\OS)}t^{\langle \rho, - \overline{ \beta^\OS_{j_r} }\rangle}}\\
&=
\mathcal{L}(\Xi (p^{\OS}_{K}) )\frac{t^{-\half}(1-t)}{1- q^{\dg(\beta_{j_r}^\OS)}t^{\langle \rho, - \overline{ \beta^\OS_{j_r} }\rangle}}.
\end{align*}
If $\dir(\ed(p^{\OS}_{K})) \xrightarrow{ - \overline{ \beta^\OS_{j_r} }^\lor} \dir(\ed(p^{\OS}_{J}))$ is a quantum edge,
then $J^{-}=K^{-}\sqcup \{j_r\}$ and $J^+=K^+$.
Therefore, we see that
\begin{align*}
\mathcal{L}(p^{\OS}_{J})&=
 t^{-\half\#J}(1-t)^{\# J}
 \prod_{j \in J^+}\frac{1}{1- q^{\dg(\beta_j^\OS)}t^{\langle \rho, - \overline{ \beta^\OS_j }\rangle}}
\prod_{j \in J^-}\frac{q^{\dg(\beta_j^\OS)}t^{\langle \rho, - \overline{ \beta^\OS_j }\rangle}}{1- q^{\dg(\beta_j^\OS)}t^{\langle \rho, - \overline{ \beta^\OS_j }\rangle}}\\
&=
 t^{-\half\# K}(1-t)^{\# K}
 \prod_{j \in K^+}\frac{1}{1- q^{\dg(\beta_j^\OS)}t^{\langle \rho, - \overline{ \beta^\OS_j }\rangle}}
\\
&\hspace{2cm}\times 
\left( \prod_{j \in K^-}\frac{q^{\dg(\beta_j^\OS)}t^{\langle \rho, - \overline{ \beta^\OS_j }\rangle}}{1- q^{\dg(\beta_j^\OS)}t^{\langle \rho, - \overline{ \beta^\OS_j }\rangle}}
\right)
\frac{q^{\dg(\beta_{j_r}^\OS)}t^{\langle \rho, - \overline{ \beta^\OS_{j_r} }\rangle-\half}(1-t)}{1- q^{\dg(\beta_{j_r}^\OS)}t^{\langle \rho, - \overline{ \beta^\OS_{j_r} }\rangle}}\\
&=
\mathcal{L}(\Xi (p^{\OS}_{K}) )\frac{q^{\dg(\beta_{j_r}^\OS)}t^{\langle \rho, - \overline{ \beta^\OS_{j_r} }\rangle-\half}(1-t)}{1- q^{\dg(\beta_{j_r}^\OS)}t^{\langle \rho, - \overline{ \beta^\OS_{j_r} }\rangle}}.
\end{align*}
This proves Claim 2 (2).
\bqed

This completes the proof of Lemma \ref{qls_qb} by induction on $\#J$.
\end{proof}

\begin{proof}[Proof of Theorem $\ref{theorem_graded_character}$]
(1)
We know from Proposition \ref{os} that
	\begin{align*}
E_{\mu}(q,t)=
\sum_{p^{\OS}_{J} \in \B({\id} ; m_{\mu}) } 
&e^{\wt (\ed (p^{\OS}_J) )}t^{\half (\ell(\dir(\ed( p^{\OS}_{J} ))) - \ell(\dir(m_{\mu})) -  \#J )}(1-t)^{\# J} \\
& \times \prod_{j \in J^+}\frac{1}{1- q^{\dg(\beta_j^\OS)}t^{\langle \rho, - \overline{ \beta^\OS_j }\rangle}}
\prod_{j \in J^-}\frac{q^{\dg(\beta_j^\OS)}t^{\langle \rho, - \overline{ \beta^\OS_j }\rangle}}{1- q^{\dg(\beta_j^\OS)}t^{\langle \rho, - \overline{ \beta^\OS_j }\rangle}}.
	\end{align*}
Therefore, it follows from Proposition \ref{bijective} and Lemma \ref{qls_qb} that
\begin{align*}
E_{\mu}(q,t)
=
\sum_{\widetilde{\eta}\in \PQLS^\mu(\lambda)}
t^{\half(\ell(v(\mu) \lons) - \ell(w_{s-1, t_{s-1}}))}
e^{\wt(\pr( \widetilde{\eta}))}
\mathcal{R}(\widetilde{\eta}).
\end{align*}
Thus, Theorem \ref{theorem_graded_character} (1) is proved.

(2)
We know from Proposition \ref{sos} that
\begin{align*}
P_{\lambda}(q,t)=
\sum_{w \in W^S}
& \sum_{p^{\OS}_{J} \in \B(w ; m_{\lambda}) } 
e^{\wt (\ed (p^{\OS}_J) )}t^{\half (\ell(\dir(\ed( p^{\OS}_{J} )))- \ell(w \dir(m_{\lambda}))  - \#J ) }(1-t)^{\# J} \\
& \times \prod_{j \in J^+}\frac{1}{1- q^{\dg(\beta_j^\OS)}t^{\langle \rho, - \overline{ \beta^\OS_j }\rangle}}
\prod_{j \in J^-}\frac{q^{\dg(\beta_j^\OS)}t^{\langle \rho, - \overline{ \beta^\OS_j }\rangle}}{1- q^{\dg(\beta_j^\OS)}t^{\langle \rho, - \overline{ \beta^\OS_j }\rangle}}.
\end{align*}
Therefore, it follows from Proposition \ref{bijective} and Lemma \ref{qls_qb} that
\begin{align*}
P_{\lambda}(q,t)&=
\sum_{w \in W^S}
\sum_{\widetilde{\eta} \in  w(\PQLS^\lambda(\lambda)) }
t^{\half(\ell(w \lons )- \ell(w_{s-1, t_{s-1}}))}
e^{\wt(\pr(\widetilde{\eta}))}\mathcal{R}(\widetilde{\eta})\\
&=
\sum_{\widetilde{\eta} \in \bigsqcup_{w \in W^S} w(\PQLS^\lambda(\lambda)) }
t^{\half(\ell(w \lons )- \ell(w_{s-1, t_{s-1}}))}
e^{\wt(\pr(\widetilde{\eta}))}\mathcal{R}(\widetilde{\eta}).
\end{align*}
Thus, Theorem \ref{theorem_graded_character} (2) is proved.
\end{proof}

\section{Pseudo-crystal structure on $\BPQLS(\lambda)$} \label{pseudo-cryst}

\subsection{Pseudo-crystals}

\begin{defi} \label{dfn:pc}
A pseudo-crystal is a set $\BB$, equipped with the maps 
$\wt:\BB \rightarrow P$, $e_{\alpha},\,f_{\alpha}:\BB \sqcup \{\zero\} 
\rightarrow \BB \sqcup \{\zero\}$ for $\alpha \in \Delta$, and 
$\ve_{\alpha},\,\vp_{\alpha}:\BB \rightarrow \BZ \sqcup \{-\infty\}$
for $\alpha \in \Delta$, satisfying the following conditions, 
where $\zero$ is an extra element not contained in $\BB$ such that 
$e_{\alpha}\zero = f_{\alpha}\zero = \zero$ for all $\alpha \in \Delta$, 
and $-\infty$ is an extra element not contained in $\BZ$ such that 
$a+(-\infty) = (-\infty) + a = -\infty$ for all $a \in \BZ$: 
\begin{enu}
\item[(i)] for $b \in \BB$ and $\alpha \in \Delta$, 
$\vp_{\alpha}(b) = \ve_{\alpha}(b) + \pair{\wt(b)}{\alpha^{\vee}}$ 
; 

\item[(ii)] for $b \in \BB$ and $\alpha \in \Delta$, 
if $e_{\alpha}b \ne \zero$, then $\wt(e_{\alpha}b)=\wt(b)+\alpha$; 

\item[(iii)] for $b \in \BB$ and $\alpha \in \Delta$, 
if $f_{\alpha}b \ne \zero$, then $\wt(f_{\alpha}b)=\wt(b)-\alpha$; 

\item[(iv)] for $b \in \BB$ and $\alpha \in \Delta$, 
if $e_{\alpha}b \ne \zero$, then $\ve_{\alpha}(e_{\alpha}b)=\ve_{\alpha}(b)-1$, 
$\vp_{\alpha}(e_{\alpha}b)=\vp_{\alpha}(b)+1$; 

\item[(v)] for $b \in \BB$ and $\alpha \in \Delta$, 
if $f_{\alpha}b \ne \zero$, then $\ve_{\alpha}(f_{\alpha}b)=\ve_{\alpha}(b)+1$, 
$\vp_{\alpha}(f_{\alpha}b)=\vp_{\alpha}(b)-1$; 

\item[(vi)] for $b,\,b' \in \BB$ and $\alpha \in \Delta$, 
$f_{\alpha}b = b'$ if and only if $b = e_{\alpha} b'$; 

\item[(vii)] for $b \in \BB$, 
if $\vp_{\alpha}(b) = -\infty$, then $e_{\alpha}b = f_{\alpha}b = \zero$. 
\end{enu}
\end{defi}

\begin{defi} \label{dfn:pcg}
The pseudo-crystal graph of a pseudo-crystal $\BB$ is the directed graph 
with vertex set $\BB$ and directed edges labeled by all roots: 
for $b,b' \in \BB$, and $\alpha \in \Delta$, 
$b \xrightarrow{\alpha} b'$ if $b'=f_{\alpha}b$. 
A pseudo-crystal is said to be connected if its pseudo-crystal graph is connected. 
\end{defi}

\subsection{Root operators on $\BPQLS(\lambda)$} 

We identify an element
$\eta = (w_1, \ldots, w_s; \sigma_0, \ldots, \sigma_s)\in \BPQLS(\lambda)$
with
the following piecewise linear, continuous map $\eta : [0,1] \rightarrow \mathfrak{h}^*_\mathbb{R}$:
\begin{equation*}
\eta(t)= \sum_{k =1}^{p-1} (\sigma_k - \sigma_{k-1})w_k \lambda + (t-\sigma_{p-1})w_p \lambda
\ 
\mbox{ for }
\sigma_{p-1} \leq t \leq \sigma_p, 1 \leq p \leq s.
\end{equation*}

\begin{lemm}[{cf. \cite[Lemma~4.5(a)]{L2}}] \label{lem:wt}
For $\eta \in \BPQLS(\lambda)$, $\wt(\eta) := \eta(1) \in P$. 
\end{lemm}

\begin{proof}
Let $\eta = (w_1, \ldots, w_s; \sigma_0, \ldots, \sigma_s) \in \BPQLS(\lambda)$. 
We see that 
\begin{equation*}
\eta(1)
=
\sum_{u=1}^{s}(\sigma_u -\sigma_{u-1})w_u \lambda = \underbrace{\sigma_s w_s \lambda}_{=w_s \lambda} + 
\sum_{u=1}^{s-1}\underbrace{\sigma_u(w_{u+1}\lambda - w_u\lambda)}_{ \text{\rm $\in P$ by Remark \ \ref{rem:1.3}} } \in P.
\end{equation*}
This proves the lemma. 
\end{proof}

We will define the root operator
$e_\alpha$ for a root $\alpha \in \Delta$
(not necessarily simple nor positive)
as in \cite{L1}.
Fix $\alpha \in \Delta$.
Let $\eta = (w_1, \ldots, w_s; \sigma_0, \ldots, \sigma_s) 
\in \BPQLS(\lambda)$. We set
\begin{align*}
H(t)  = H^\eta_\alpha (t) := \langle \eta(t), \alpha^\lor \rangle \ \ \mbox{for } t \in [0,1], \\
m= m^\eta_\alpha
:= 
\min \left( 
\{
H^\eta_\alpha (t)
\ | \
t \in [0,1]
\}
\cap\mathbb{Z}
\right);
\end{align*}
note that
$m \in \mathbb{Z}_{\geq 0}$
since $H(0)=0$.
If $m =0$, then we set $e_\alpha \eta := \zero$.
If $m \leq -1$, then we define $0 \leq t_0 < t_1 \leq 1$
by
\begin{equation}\label{eq_3.1}
t_1 := \min \{ t \in [0,1] \mid H(t)= m \},
\end{equation}
\begin{equation}\label{eq_3.2}
t_0 := \max
\{t \in [0, t_1] \mid H(t)= m+1 \};
\end{equation}
observe that $t_0$ is a unique point in $[0,t_1]$
such that
\begin{equation*}
H(t_0) = m+1 \ \mbox{ and } \ m < H(t) < m+1 
\mbox{ for } t_0 < t <t_1.
\end{equation*}
Now we define
\begin{equation} \label{eq:ea}
(e_\alpha \eta)(t) := 
\begin{cases}
\eta(t )& \mbox{ for }0 \leq t \leq t_0, \\
\eta(t_0)+ s_\alpha(\eta(t)-\eta(t_0))& \mbox{ for }t_0 \leq t \leq t_1, \\
\eta(t) + \alpha & \mbox{ for }t_1\leq t \leq 1. 
\end{cases}
\end{equation}

\begin{rem} \label{rem:ealpha}
Let $0 \leq p \leq q \leq s-1$ be such that $\sigma_{p-1} \leq t_0 < \sigma_{p}$
and $\sigma_{q-1} < t_1 \leq \sigma_{q}$; note that $p \le q$. Then, 
\begin{align}
& e_\alpha \eta = 
(w_1, \ldots , w_p, \lfloor s_\alpha w_{p}\rfloor, \lfloor  s_{\alpha} w_{p+1}\rfloor,  \ldots , \lfloor  s_\alpha w_{q}\rfloor,  w_{q}, w_{q+1},\ldots , w_s; \nonumber \\ 
& \hspace*{20mm}
\sigma_0, \ldots, \sigma_{p-1}, t_0, \sigma_p, \ldots, \sigma_{q-1}, t_1, \sigma_{q}, \ldots , \sigma_s ).  \label{eq_3.3}
\end{align}
Here, if $t_{0}=\sigma_{p-1}$, then we drop $w_{p}$ and $\sigma_{p-1}$; 
if $t_{1}=\sigma_{q}$ (and $\lfloor s_\alpha w_{q} \rfloor \ne w_{q+1}$), 
then we drop $w_{q}$ and $\sigma_{q}$; 
if $t_{1}=\sigma_{q}$ and $\lfloor s_\alpha w_{q} \rfloor = w_{q+1}$, 
then we drop $w_{q+1}$ and $t_{1}$, $\sigma_{q}$. 
\end{rem}

\begin{pro}\label{3.3}
Let $\eta \in \BPQLS(\lambda)$, and $\alpha \in \Delta$.
If $e_\alpha \eta \neq \zero$, then $e_\alpha \eta \in \BPQLS(\lambda)$ and 
$\wt(e_\alpha \eta) = \wt(\eta)+\alpha$. 
\end{pro}

\begin{proof}
Keep the setting of Remark~\ref{rem:ealpha}. 
We give a proof in the case that $t_1 \neq \sigma_q$ and $t_0 \neq \sigma_{p-1}$;
the proofs in the other cases are similar.
By \eqref{eq_3.3} and Lemma \ref{parabolic_DBG}, it suffices to show that 
\begin{enu}
\item[(i)]
there exists a path from $s_\alpha w_p$ to $w_p$ in $\DBG_{t_0 \lambda}$;

\item[(ii)]
for each $p \leq u \leq q-1$,
there exists a path from $s_\alpha w_{u+1}$ to $s_\alpha w_u$ in $\DBG_{\sigma_u \lambda}$;

\item[(iii)]
there exists a path from $w_q$ to $s_\alpha w_q$ in $\DBG_{t_1 \lambda}$.
\end{enu}
Let us show (i). We have
\begin{equation} \label{eq:33a}
\eta(t_0)
=
\sum_{u=1}^{p-1}(\sigma_u -\sigma_{u-1})w_u \lambda + (t_0 - \sigma_{p-1}) w_p \lambda = t_0 w_p \lambda + 
\sum_{u=1}^{p-1}\underbrace{\sigma_u(w_{u+1}\lambda - w_u\lambda)}_{ \text{\rm $\in P$ by Remark \ \ref{rem:1.3}} }.
\end{equation}
Since
$\langle \eta(t_0), \alpha^\lor  \rangle = H(t_0) = m+1 \in \mathbb{Z}$,
it follows that 
$\langle t_0 w_p \lambda, \alpha^\lor \rangle \in \mathbb{Z}$,
and hence $t_0 \langle  \lambda, w_p^{-1}\alpha^\lor  \rangle \in \mathbb{Z}$; 
by the maximality of $t_0$, we see that $\langle w_p \lambda, \alpha^\lor \rangle = 
\langle \lambda, w_p^{-1}\alpha^\lor \rangle < 0$, which implies that
$w_p^{-1}\alpha \in \Delta^- \setminus \Delta^-_S$. 
Thus we obtain $w_p \xleftarrow{ -w_p^{-1}\alpha } s_\alpha w_p$ in $\DBG_{t_{0} \lambda}$, 
which shows (i). 
Similarly, from the fact that
$\langle \eta(t_1), \alpha^\lor \rangle = H(t_1) = m \in \mathbb{Z}$,
we deduce that $s_\alpha w_q \xleftarrow{-w_q^{-1}\alpha} w_q$ in $\DBG_{t_1 \lambda}$,
which shows (iii).
Part (ii) follows immediately from the definition of $\DBG_{\sigma_u \lambda}$. 
Indeed, if $w_u = w_{u, 0} \xleftarrow{\beta_{u,1}}w_{u, 1}\xleftarrow{\beta_{u,2}} \cdots 
\xleftarrow{\beta_{u,k_u}} w_{u, k_u} = w_{u+1}$
is a directed path from $w_{u+1}$ to $w_u$ in $\DBG_{\sigma_u \lambda}$, then 
\begin{equation*}
s_\alpha w_u =s_\alpha  w_{u, 0} \xleftarrow{\beta_{u,1}}s_\alpha w_{u, 1}\xleftarrow{\beta_{u,2}} \cdots \xleftarrow{\beta_{u,k_u}} s_\alpha w_{u, k_u} = s_{\alpha}w_{u+1}
\end{equation*}
is a directed path from $s_\alpha w_{u+1}$ to $s_\alpha w_u$ in $\DBG_{\sigma_u \lambda}$. 
The equality $\wt(e_\alpha \eta) = \wt(\eta)+\alpha$ follows from \eqref{eq:ea}. 
This proves the proposition.
\end{proof}

We can easily verify the following lemma 
by using the definition of root operators.
\begin{lemm} \label{3.1}
Let $\eta \in \BPQLS(\lambda)$, and $\alpha \in \Delta$.
Assume that $m^\eta_\alpha \leq -1$,
and define $0 \leq t_0 < t_1 \leq 1$ as in \eqref{eq_3.1}, \eqref{eq_3.2} 
for the $\eta$ and $\alpha$.
Then, $m^{e_\alpha \eta}_\alpha = m^\eta_\alpha +1$,
and $H^{e_\alpha \eta}_\alpha (t_0) = m^\eta_\alpha +1 = m^{e_\alpha \eta}_\alpha$,
which implies that $t'_1 := \min \{ t \in [0,1] \mid H^{e_\alpha \eta}_\alpha (t) = 
m^{e_\alpha \eta}_\alpha \}$ is less that or equal to $t_0$.
\end{lemm}

The following is an immediate consequence of Lemma~\ref{3.1}. 
\begin{cor}\label{3.2}
Let $\eta \in \BPQLS(\lambda)$, and $\alpha \in \Delta$.
\begin{enu}
\item If we set $\ve_\alpha (\eta) := \max \{ n \geq 0 \mid e^n_\alpha \eta \neq \zero \}$,
then we have $\ve_\alpha (\eta) = - m_\alpha^\eta$ {\rm (cf. \cite[Proposition (a) in \S1.5]{L1} 
and \cite[Lemma~2.1 (c)]{L2})}.

\item
For each $1 \leq k \leq N := \varepsilon_\alpha (\eta) -1 $,
define $0 \leq t_0^{(k)} < t_1^{(k)}\leq 1$
as in {\rm(\ref{eq_3.1})} and {\rm(\ref{eq_3.2})} for $e_\alpha^k \eta$ and $\alpha$.
Then, we have 
\begin{equation*}
0 \leq 
\underbrace{t_0^{(N)} < t_1^{(N)}}_{ \text{\rm for $e_\alpha^N \eta$} }
\leq
\underbrace{t_0^{(N-1)} < t_1^{(N-1)}}_{ \text{\rm for $e_\alpha^{N-1} \eta$} }
\leq \cdots \leq
\underbrace{t_0^{(1)} < t_1^{(1)}}_{ \text{\rm for $e_\alpha \eta$} }
\leq
\underbrace{t_0^{(0)} < t_1^{(0)}}_{ \text{\rm for $\eta$} }
\leq 1,
\end{equation*}
with
\begin{align*}
t_1^{(k)}  &= \min\{ 
t \in [0,1] \mid 
H^{\eta}_{\alpha} (t) = m_\alpha^\eta +k
\},
\\
t_0^{(k)}  &= \max\{ 
t \in [0,t_1^{(k)}] \mid 
H^{\eta}_{\alpha} (t) = m_\alpha^\eta +k+1
\};
\end{align*}
note that $t_0^{(k)}$ is a unique point in $[0, t_1^{(k)}]$ such that
\begin{equation*}
H_\alpha^\eta(t_0^{(k)}) = m^\eta_\alpha +k+1,
\ \text{\rm and} \ 
m_\alpha^\eta+k
<
H_\alpha^\eta (t)
<
m_\alpha^\eta+k+1
\ \text{\rm for} \ 
t_0^{(k)} <t <t_1^{(k)}.
\end{equation*}
\end{enu}
\end{cor}

Similarly, we will define the root operator
$f_\alpha$ for $\alpha \in \Delta$.
Let $\eta \in \BPQLS(\lambda)$, and $\alpha \in \Delta$. 
If $H(1)=m$, then we set $f_\alpha \eta := \zero$.
If $H(1)-m \ge -1$, then we define
$0 \leq t_0 < t_1 \leq 1$ by
\begin{equation}\label{eq_3.1f}
t_0 := \min \{ t \in [0,1] \mid H(t)= m \},
\end{equation}
\begin{equation}\label{eq_3.2f}
t_1 := \max
\{t \in [t_0, 1] \mid H(t)= m+1 \}.
\end{equation}
Now we define
\begin{equation} \label{eq:fa}
(f_\alpha \eta)(t) :=
\begin{cases}
\eta(t )& \mbox{ for }0 \leq t \leq t_0, \\
\eta(t_0)+ s_\alpha(\eta(t)-\eta(t_0))& \mbox{ for }t_0 \leq t \leq t_1, \\
\eta(t) - \alpha & \mbox{ for }t_1\leq t \leq 1. 
\end{cases}
\end{equation}
In the same way as for 
Proposition \ref{3.3}, we can show that
if $f_\alpha \eta\neq \zero$,
then $f_\alpha \eta \in \BPQLS(\lambda)$ and $\wt(f_\alpha \eta)=\wt(\eta)-\alpha$. 
Also, we deduce that 
$\vp_\alpha (\eta) := \max \{ n \geq 0 \mid f_\alpha^{n} \eta \neq \zero \}$ 
is equal to $H_{\alpha}^{\eta}(1) - m_\alpha^\eta$. 
The proof of the following proposition is
the same as that of \cite[Proposition (b) in \S1.5]{L1}.
\begin{pro}
Let $\eta \in \BPQLS (\lambda)$, and $\alpha \in \Delta$. 
If $e_\alpha \eta \neq \zero$,
then $f_\alpha e_\alpha \eta = \eta$. 
If $f_\alpha \eta \neq \zero$,
then $e_\alpha f_\alpha \eta = \eta$.
\end{pro}

\begin{theo}
The set $\BPQLS(\lambda)$, equipped with the maps 
$\wt:\BPQLS(\lambda) \rightarrow P$, 
$e_{\alpha},\,f_{\alpha}:\BPQLS(\lambda) \sqcup \{\zero\} 
\rightarrow \BPQLS(\lambda) \sqcup \{\zero\}$ for $\alpha \in \Delta$, and 
$\ve_{\alpha},\,\vp_{\alpha}:\BPQLS(\lambda) \rightarrow \BZ_{\ge 0}$ 
for $\alpha \in \Delta$, becomes a pseudo-crystal. 
\end{theo}

\subsection{Connectedness of the pseudo-crystal $\BPQLS(\lambda)$}

Fix a dominant weight $\lambda$.
We set $\eta_w := (w; 0,1)$ for $w \in W$.
\begin{theo}\label{4.1}
For all $\eta \in \BPQLS(\lambda)$, 
there exist
$\zeta_1, \ldots, \zeta_k \in \Delta$ such that
$e_{\zeta_1}\cdots e_{\zeta_k}\eta = \eta_\id$. 
In particular, the pseudo-crystal $\BPQLS(\lambda)$ is connected. 
\end{theo}

In order to prove this theorem,
we need a few lemmas given below. For
$\eta \in \BPQLS(\lambda)$ and $\alpha \in \Delta$,
we set
$e_\alpha^{\max} := e_\alpha^{\varepsilon_\alpha(\eta)}\eta \in \BPQLS(\lambda)$;
recall that $\ve_\alpha (\eta) = -m_\alpha^\eta$.

\begin{lemm}\label{4.2}
Let
$\eta = (w_1 , \ldots , w_s; \sigma_0 , \ldots, \sigma_s )\in \BPQLS(\lambda)$, and $\alpha \in \Delta$.
Assume that 
$H_\alpha^\eta (\sigma_1) = \sigma_1 \langle w_1 \lambda, \alpha^\lor \rangle \in \mathbb{Z}_{< 0}$.
Set
$b := - H_\alpha^\eta(\sigma_1) \in \mathbb{Z}_{> 0}$,
$m := m_\alpha^\eta$, and
$p := -m -b \in \mathbb{Z}_{\geq 0}$.
For each
$0 \leq k \leq \varepsilon_\alpha(\eta)-1=-m -1=p+b-1$,
define
$0 \leq t_0^{(k)}<t_1^{(k)}\leq 1$
as in \eqref{eq_3.1} and \eqref{eq_3.2} for
$e_\alpha^k \eta \in \BPQLS(\lambda)$ and $\alpha$. Then, we have 
\begin{align*}
0=
\underbrace{t_0^{(p+b-1)}<t_1^{(p+b-1)}}_{ \text{\rm for $e_\alpha^{p+b-1}\eta$} }
&=\cdots =
\underbrace{t_0^{(p+1)}<t_1^{(p+1)}}_{ \text{\rm for $e_\alpha^{p+1}\eta$} }
=
\underbrace{t_0^{(p)}<t_1^{(p)}}_{ 
 \text{\rm for $e_\alpha^p \eta$} }=\sigma_1\\
&\leq 
\underbrace{t_0^{(p-1)} < t_1^{(p-1)}}_{ 
 \text{\rm for $e_\alpha^{p-1}\eta$} } \le \cdots \le
\underbrace{t_0^{(1)} < t_1^{(1)}}_{ 
 \text{\rm for $e_\alpha \eta$} } \le 
\underbrace{t_0^{(0)} < t_1^{(0)}}_{ 
 \text{\rm for $\eta$} }\leq 1,
\end{align*}
with
\begin{equation}\label{eq_4.1}
t_1^{(p+k)}=
\frac{(b-k)\sigma_1}{b},
\quad
t_0^{(p+k)}=
\frac{(b-k-1)\sigma_1}{b}
\quad \text{\rm for} \ 
0 \leq k \leq b-1. 
\end{equation}
Hence, for $t \in [0,\sigma_1]$, we have
\begin{equation}\label{eq_4.2}
(e_\alpha^{p+k}\eta)(t)=
\begin{cases}
t(w_1 \lambda) & \text{\rm for } t \in [0, \frac{(b-k)\sigma_1}{b}],\\[2mm]
t(s_\alpha w_1 \lambda)& \text{\rm for } t \in [\frac{(b-k)\sigma_1}{b}, \sigma_1],
\end{cases}
\quad \text{\rm for } 0 \leq k \leq b-1; 
\end{equation}
in particular, 
\begin{equation*}
(e_\alpha^{\max}\eta)(t) = t(s_\alpha w_1 \lambda) \quad 
\text{\rm for } t \in [0, \sigma_1].
\end{equation*}
\end{lemm}

\begin{proof}
First we claim that
\begin{equation}\label{eq_4.3}
(e_\alpha^{p}\eta)(t) = \eta (t) = t(w_1 \lambda)
\mbox{ for }
t \in  [0, \sigma_1],
\end{equation}
which is equivalent to the inequalities:
$\sigma_1 \leq t_0^{(p-1)}<t_1^{(p-1)}\leq \cdots \leq t_0^{(1)} < t_1^{(1)} \leq t_0^{(0)} < t_1^{(0)} \leq 1$
(see Corollary \ref{3.2}).
Suppose, for a contradiction, that
$(e_\alpha^p \eta)(t)\neq \eta (t)$ for some $t \in [0, \sigma_1]$.
Let $1 \leq q \leq p$
be the minimal element such that $(e_\alpha^q \eta)(t) \neq \eta(t)$
for some $t \in [0, \sigma_1]$.
Since 
$(e_\alpha^{q-1}\eta)(t)= \eta (t)$
for all $t \in [0, \sigma_1]$
by the definition of $q$,
we see from the definition of $e_\alpha$
that
$t_0^{(q-1)}<\sigma_1$.
Also, since
\begin{equation*}
H_\alpha^\eta (t) = 
t \underbrace{\langle w_1 \lambda, \alpha^\lor \rangle}_{<0}
\mbox{ for }
t \in [0, \sigma_1],
\end{equation*}
it follows that 
$H_\alpha^\eta (t_0^{(q-1)}) > H_\alpha^\eta(\sigma_1)$.
However,
by Corollary \ref{3.2},
we have
\begin{equation*}
H_\alpha^\eta (t_0^{(q-1)}) = m_\alpha^\eta + (q-1)+1 =m+q
\leq m+p = H_\alpha^\eta (\sigma_1),
\end{equation*}
which is a contradiction.
Thus we have shown \eqref{eq_4.3}.
Now it follows from Corollary~\ref{3.2}\,(1) that 
\begin{equation*}
- m^{e_{\alpha}^{p}\eta}_{\alpha} = \ve_{\alpha}(e_{\alpha}^{p}\eta) = 
\ve_{\alpha}(\eta) - p = - m - p = b. 
\end{equation*}
By \eqref{eq_4.3}, together with the assumption that 
$\sigma_1 \langle w_1 \lambda, \alpha^\lor \rangle \in \mathbb{Z}_{< 0}$, 
we see that the function $H^{e_{\alpha}^{p}\eta}_{\alpha}(t)$ is 
strictly decreasing on $[0,\sigma_{1}]$, and that
$H^{e_{\alpha}^{p}\eta}_{\alpha}(\sigma_{1}) = -b = m^{e_{\alpha}^{p}\eta}_{\alpha}$. 
Hence we have $t_{1}^{(p)}=\sigma_{1}$ and $t_{0}^{(p)}=(b-1)\sigma_{1}/b$ 
by the definition of these points. Also, from the definition of the root operator $e_{\alpha}$ and 
\eqref{eq_4.3}, we deduce that for $t \in  [0, \sigma_{1}]$, 
\begin{equation} \label{eq_4.3-2}
(e_\alpha^{p+1}\eta)(t) = 
\begin{cases}
t(w_1 \lambda) & \text{\rm for } t \in [0, \frac{(b-1)\sigma_1}{b}], \\[2mm]
t(s_\alpha w_1 \lambda) & \text{\rm for } t \in [\frac{(b-1)\sigma_1}{b}, \sigma_1]. 
\end{cases}
\end{equation}
The same argument as above shows that 
the function $H^{e_{\alpha}^{p+1}\eta}_{\alpha}(t)$ is 
strictly decreasing on $[0,t_{0}^{(p)}]$, and that
$H^{e_{\alpha}^{p+1}\eta}_{\alpha}(t_{0}^{(p)}) = - b + 1 = 
m^{e_{\alpha}^{p+1}\eta}_{\alpha}$. 
Hence we have $t_{1}^{(p+1)}=t_{0}^{(p)}$ and $t_{0}^{(p+1)}=(b-2)\sigma_{1}/b$ 
by the definition of these points. From the definition of the root operator $e_{\alpha}$ and 
\eqref{eq_4.3-2}, we deduce that 
\begin{equation*}
(e_\alpha^{p+2}\eta)(t) = 
\begin{cases}
t(w_1 \lambda) & \text{\rm for } t \in [0, \frac{(b-2)\sigma_1}{b}], \\[2mm]
t(s_\alpha w_1 \lambda) & \text{\rm for } t \in [\frac{(b-2)\sigma_1}{b}, \sigma_1]. 
\end{cases}
\end{equation*}
By repeating this argument, we obtain \eqref{eq_4.1} and \eqref{eq_4.2}. 
This proves the lemma.
\end{proof}

\begin{lemm}\label{4.3}
Let $\eta = (w_1, \ldots, w_s; \sigma_1 \ldots , \sigma_s) 
\in \BPQLS(\lambda)$, with $s \geq 2$.
Let $v_2 \in w_2 W_S$ be such that 
there exists a directed path
$w_1 = 
x_0
\xleftarrow{\gamma_1}
x_1
\xleftarrow{\gamma_2}
\cdots
\xleftarrow{\gamma_p}
x_p = v_2$ from $v_2$ to $w_1$ in 
$\DBG_{\sigma_1 \lambda}$
whose labels lie in $\Delta^+ \setminus \Delta^+_S$; 
see Lemma~{\rm \ref{parabolic_DBG}\,(1)}.
Then, for all $1 \leq q \leq p$, we have
$H^{\eta}_{-w_1 \gamma_q}(\sigma_1)\in \mathbb{Z}_{<0}$. 
\end{lemm}

\begin{proof}
Since $\lambda$ is a dominant weight, 
and $\gamma_q \in \Delta^+ \setminus \Delta^+_S$, 
we see from the definition of $\DBG_{\sigma_1 \lambda}$ that 
$\sigma_1 \langle \lambda ,\gamma_q ^\lor \rangle \in \mathbb{Z}_{>0}$. 
Hence it follows that 
\begin{equation*}
H_{-w_1 \gamma_q}^\eta (\sigma_1) =
\sigma_1
\langle w_1 \lambda, - w_1 \gamma_q^\lor \rangle
=
\sigma_1\langle \lambda, - \gamma_q^\lor
\rangle
\in \mathbb{Z}_{<0}.
\end{equation*}
This proves the lemma.
\end{proof}

\begin{lemm} \label{lem:sig}
Let $\eta = (w_1, \ldots, w_s; \sigma_1 \ldots , \sigma_s) 
\in \BPQLS(\lambda)$, with $s \geq 2$. 
Then, there exists a monomial $X$ in the root operators $e_{\alpha}$, 
$\alpha \in \Delta$, such that $X\eta$ is of the form: 
$(v_{1},\dots ; \sigma_{0},\sigma_{1}',\dots)$, 
with $\sigma_{1} < \sigma_{1}' \le 1$. 
\end{lemm}

\begin{proof}
As in Lemma~\ref{4.3}, 
let $v_2 \in w_2 W_S$ be such that 
there exists a directed path
$w_1 = 
x_0
\xleftarrow{\gamma_1}
x_1
\xleftarrow{\gamma_2}
\cdots
\xleftarrow{\gamma_p}
x_p = v_2$ from $v_2$ to $w_1$ in 
$\DBG_{\sigma_1 \lambda}$
whose labels lie in $\Delta^+ \setminus \Delta^+_S$.
We will prove the assertion of the lemma by induction on $p \geq 1$. 
Assume first that $p =1$. Set $\alpha = - w_1 \gamma_1 \in \Delta$;
notice that $w_2 = \lfloor w_1 s_{\gamma_1}  \rfloor = \lfloor s_\alpha w_1 \rfloor$.
Then, by Lemma \ref{4.3},
we see that
$H_\alpha^\eta (\sigma_1) = \sigma_1 \langle  w_1 \lambda, \alpha^\lor \rangle
\in \mathbb{Z}_{<0}$.
As in Lemma \ref{4.2},
we set 
$b := -H_\alpha^\eta(\sigma_1)\in \mathbb{Z}_{>0}$,
$m := m_\alpha^\eta$, and
$p := -m-b \in \mathbb{Z}_{\geq 0}$.
Also,
for each $0 \leq k \leq \varepsilon_\alpha(\eta)-1 = -m-1 = p+b-1$,
we define
$0 \leq t_0^{(k)} < t_1^{(k)}\leq 1$
as in (\ref{eq_3.1}), (\ref{eq_3.2})
for $e_\alpha^k \eta \in \BPQLS(\lambda)$ and $\alpha$.
Then, by Lemma \ref{4.2}, we see that
\begin{align*}
0=
\underbrace{t_0^{(p+b-1)}<t_1^{(p+b-1)}}_{ \text{\rm for $e_\alpha^{p+b-1}\eta$} }
&=\cdots =
\underbrace{t_0^{(p+1)}<t_1^{(p+1)}}_{ \text{\rm for $e_\alpha^{p+1}\eta$} }
=
\underbrace{t_0^{(p)}<t_1^{(p)}}_{ 
 \text{\rm for $e_\alpha^p \eta$} }=\sigma_1\\
&\leq 
\underbrace{t_0^{(p-1)} < t_1^{(p-1)}}_{ 
 \text{\rm for $e_\alpha^{p-1}\eta$} } \le \cdots \le
\underbrace{t_0^{(1)} < t_1^{(1)}}_{ 
 \text{\rm for $e_\alpha \eta$} } \le 
\underbrace{t_0^{(0)} < t_1^{(0)}}_{ 
 \text{\rm for $\eta$} }\leq 1,
\end{align*}
and that $(e_\alpha^{\max} \eta) (t) = t(s_\alpha w_1 \lambda)$
for $t \in [0,\sigma_1]$.
Since 
$\langle w_2 \lambda, \alpha^\lor \rangle = \langle s_\alpha w_1 \lambda, \alpha^\lor \rangle >0$,
it follows that 
$H_\alpha^\eta (\sigma_1 + \epsilon) > H_\alpha^\eta (\sigma_1 ) = m+ p$
for sufficiently small $\epsilon >0$.
Also, since 
$m+p-1 < H_\alpha^\eta (t)< m+p$
for $t_0^{(p-1)}< t< t_1^{(p-1)}$
(see Corollary \ref{3.2}),
we obtain
$t_0^{(p-1)}\neq \sigma_1$, and
hence $\sigma_1 < t_0^{(p-1)}$.
Therefore, we deduce that
\begin{equation*}
(e_\alpha^{\max} \eta)(\sigma_1 + \epsilon) = w_2 \lambda = s_\alpha w_1 \lambda
\end{equation*}
for sufficiently small $\epsilon >0$, 
which implies that $e_\alpha^{\max} \eta$
is of the form:
$(\lfloor s_\alpha w_1 \rfloor, \ldots; \sigma_0, \sigma'_1, \ldots)$,
with $\sigma'_1 > \sigma_1$.

Assume next that $p >1$. As above,
if we set $\alpha := - w_1 \gamma_1$,
then we see by Lemma \ref{4.3} that
$H_\alpha^\eta(\sigma_1) = \sigma_1 \langle w_1 \lambda , \alpha^\lor \rangle \in \mathbb{Z}_{<0}$,
and then by Lemma \ref{4.2} that
$(e_\alpha^{\max} \eta )(t) = t(s_\alpha w_1 \lambda)$
for $t \in [0, \sigma_1]$.
Hence it follows that 
$e_\alpha^{\max} \eta$ is of the form:
$e_\alpha^{\max} \eta = (\lfloor s_\alpha w_1 \rfloor, \ldots ; \sigma_0, \sigma'_1, \ldots )$, 
with $\sigma'_1 \geq \sigma_1$.
If $\sigma'_1 > \sigma_1$,
then the assertion is obvious. 
Assume now that $\sigma'_1 = \sigma_1$. 
Then, $e_\alpha^{\max} \eta$ is of the form:
\begin{equation*}
e_\alpha^{\max} \eta = (\lfloor s_\alpha w_1 \rfloor, w'_2 \ldots ; \sigma_0, \sigma_1, \ldots);
\end{equation*}
notice that $w'_2$ is equal to either $w_2$ or $\lfloor s_\alpha w_2 \rfloor$
by the definition of root operator $e_\alpha$.
Assume that $w'_2 = w_2$.
Since
\begin{equation*}
s_\alpha w_1 = w_1 s_{\gamma_1} = x_1
\xleftarrow{\gamma_2} \cdots \xleftarrow{\gamma_p} x_p = v_2
\end{equation*}
is a directed path (of length $p-1$) from $v_2 \in w_2 W_S = w_{2}' W_{S}$ 
to $s_\alpha w_1$ in $\DBG_{\sigma_1 \lambda}$
whose labels lie in $\Delta^+ \setminus \Delta^+_S$,
the assertion follows from our induction hypothesis (applied to $e_\alpha^{\max} \eta$).
Assume that $w'_2 = \lfloor s_\alpha w_2 \rfloor$. 
If we set
$p := H_\alpha^\eta(\sigma_1)- m^\eta_\alpha$,
then we deduce from Lemma \ref{4.2}
that $e_\alpha^p \eta$ is of the form:
\begin{equation*}
e_\alpha^p \eta = (w_1, w'_2, \ldots; \sigma_0, \sigma_1, \ldots).
\end{equation*}
Since $s_\alpha v_2
= s_{- w_1 \gamma_1} w_1 s_{\gamma_1} s_{\gamma_2 }\cdots s_{\gamma_p}
=w_1 s_{\gamma_2 }\cdots s_{\gamma_p}$, it follows that
\begin{equation*}
w_1
\xleftarrow{\gamma_2}
w_1 s_{\gamma_2}
\xleftarrow{\gamma_3}
\cdots
\xleftarrow{\gamma_p}
w_1 s_{\gamma_2 }\cdots s_{\gamma_p}
=  s_\alpha v_2
\end{equation*}
is a directed path (of length $p-1$)
from $s_\alpha  v_2 \in w'_2 W_S$ to $w_1$
in $\DBG_{\sigma_1 \lambda}$
whose labels lie in $\Delta^+ \setminus \Delta^+_S$.
Hence the assertion follows from our induction hypothesis 
(applied to $e_\alpha^{p} \eta$). 
This completes the proof of the lemma.
\end{proof}


\begin{proof}[Proof of Theorem~{\rm\ref{4.1}}]
We set
\begin{equation*}
T_\lambda
:= 
\{
b \in [0,1]\cap \mathbb{Q} \mid 
b \langle \lambda, \alpha^\lor \rangle \in \mathbb{Z}
\mbox{ for some }
\alpha \in \Delta 
\};
\end{equation*}
notice that $T_\lambda$ is a finite set. 
Let $\eta = (w_1 ,\ldots , w_s; \sigma_0 , \ldots, \sigma_s )\in \BPQLS(\lambda)$; 
remark that $\sigma_u \in T_\lambda$ for all $0 \leq u \leq s$.
We prove the assertion of the theorem by (descending) induction on $\sigma_1 \in T_\lambda$.
Assume that $\sigma_1 =1$; note that $\eta$ is of the form
$\eta = \eta_{w} = (w; 0,1)$. We show the assertion in this case by induction on 
the length $\ell(w)$ of $w \in W^{S}$. If $\ell(w) = 0$, i.e., if $w = e$, 
then there is nothing to show. Assume that $\ell(w) > 0$. 
Let us take $i \in I$ such that 
$\pair{w\lambda}{\alpha_{i}^{\vee}} < 0$; note that $s_{i}w \in W^{S}$ in this case.
Then, by Lemma~\ref{4.2}, we deduce that 
$e_{\alpha_{i}}^{\max}\eta = e_{\alpha_{i}}^{\max}\eta_{w} = (s_{i}w; 0,1)$. 
Since $\ell(s_{i}w) < \ell(w)$, the assertion follows by the induction hypothesis. 
Assume now that $\sigma_1<1$;
notice that $s \geq 2$.
By Lemma~\ref{lem:sig}, 
there exists a monomial $X$ in the root operators $e_{\alpha}$, 
$\alpha \in \Delta$, such that $X\eta$ is of the form: 
$(v_{1},\dots ; \sigma_{0},\sigma_{1}',\dots)$, 
with $\sigma_{1} < \sigma_{1}' \le 1$. 
The assertion of the theorem follows from our induction hypothesis (applied to $X\eta$). 
This completes the proof of Theorem~\ref{4.1}. 
\end{proof}

\end{document}